\documentclass[10pt]{amsart}
\usepackage{latexsym}
\usepackage{stmaryrd}
\usepackage{amsmath}
\usepackage{epsfig}
\usepackage{graphicx}
\usepackage{eufrak}
\usepackage{dsfont}
\usepackage[english]{babel}
\usepackage{color}
\usepackage{mathabx}
\usepackage{amscd}      
\usepackage{amssymb}
\usepackage{dsfont}
\usepackage{xypic}      
\LaTeXdiagrams          
\usepackage[all,v2]{xy}
\xyoption{2cell} \UseAllTwocells \xyoption{frame} \CompileMatrices
\usepackage{latexsym}
\usepackage{epsfig}
\usepackage{enumerate}

\usepackage{latexsym}
\usepackage{epsfig}
\usepackage{amsfonts}
\usepackage{enumerate}
\usepackage{mathrsfs}

\allowdisplaybreaks[3]

\newtheorem{prop}{Proposition}[section]
\newtheorem{lem}[prop]{Lemma}

\newtheorem{cor}[prop]{Corollary}
\newtheorem{thm}[prop]{Theorem}
\newtheorem{rmk}[prop]{Remark}
\newtheorem{example}{Example}

\newtheorem{defn}[prop]{Definition}

            {\nolinebreak $\Box$ \end{trivlist}}

\newcommand{\noprint}[1]{}

\renewcommand{\tilde}{\widetilde}

\newcommand{\qq}{{\mathbb Q}}

\newcommand{\sI}{{\mathcal I}}

\newcommand{\sO}{{\mathcal O}}

\newcommand{\sX}{{\mathcal X}}

\newcommand{\Coh}{\mbox{Coh}}

\newcommand{\cEnd}{\mathscr{E}nd}
\newcommand{\cHom}{\mathscr{H}om}

\DeclareMathOperator{\id}{id}

\DeclareMathOperator{\Ch}{Ch}
\DeclareMathOperator{\CR}{CR}

\DeclareMathOperator{\Supp}{Supp}

\DeclareMathOperator{\Td}{Td}

\newcommand{\rk}{\mathop{\rm rk}}

\newcommand{\pr}{\mathop{\rm pr}\nolimits}

\newcommand{\red}{\mathop{\rm red}\nolimits}

\numberwithin{equation}{subsection}

\newcommand {\mat}      [1] {\left(\begin{array}{#1}}
\newcommand {\rix}          {\end{array}\right)}

\title[Fixed point locus of Moduli spaces of Sheaves on Toric DM stacks]{Fixed point locus of Moduli spaces of sheaves on Toric Dm stacks}
\author{ Promit Kundu}
\address{Institute of Mathematics\\ Shanghai Tech University\\ 393 Middle Huaxia Road\\ Pudong, Shanghai, 201210, China} 
\email{kundupromit63@gmail.com}

\begin{document}
\sloppy 
\begin{abstract}
Extending work of Klyachko, Perling and Kool we develop a combinatorial description of torsion free toric sheaves in any dimension on smooth toric DM stacks. We investigate their basic properties and under certain conditions recover some known results on smooth toric varieties. The action of the torus lifts to the moduli of torsion free modified Gieseker stable sheaves on the smooth DM stack, and we express its fixed point locus explicitly in terms of certain finer invariants called characteristic functions. These techniques will be exploited to compute topological invariants of the moduli of modified stable torsion free sheaves on smooth toric DM stacks. 
\end{abstract}

\maketitle

\tableofcontents

\section{Introduction}
In this paper, we study the moduli space of torsion-free sheaves on smooth projective toric Deligne–Mumford (DM) stacks. A central theme in the moduli theory of sheaves on smooth projective DM stacks is the interplay between geometric invariant theory (GIT) and stability conditions arising from an ample line bundle on the coarse moduli space and a vector bundle on the stack. We explore the moduli of modified stable torsion-free toric sheaves on smooth toric DM stacks, which correspond to fixed points of the natural torus action on the moduli space of modified Gieseker stable sheaves. This correspondence allows us to obtain a decomposition of the moduli space of stable toric torsion-free sheaves with a fixed Hilbert polynomial in terms of moduli spaces of stable torsion-free sheaves with fixed characteristic functions that give rise to the same modified Hilbert polynomial. Such a decomposition reduces the study of a complicated moduli space to a disjoint union of better-understood components parametrized by combinatorial data.

When studying Gieseker and slope stability conditions for torsion-free toric sheaves on smooth DM stacks, we must invoke the concept of a generating sheaf $\Xi$ as introduced in \cite{Nir}. The generating sheaf is an ample vector bundle on a smooth DM stack that plays a crucial role in defining a Hilbert polynomial for a coherent sheaf on a stack. Without the generating sheaf, the Hilbert polynomial defined solely in terms of the ample line bundle coming from the projective coarse moduli space only captures the behavior of sheaves on the coarse moduli space, losing information about the stacky structure. To address this, our Hilbert polynomial is modified to incorporate both the ample line bundle and the generating sheaf, and we study stability with respect to an ample line bundle together with a generating sheaf equipped with a fixed toric equivariant structure. This modification ensures that the stability condition reflects the genuine stacky geometry.

We now introduce our main theorem. Let $\sX$ be a smooth toric DM stack with top cones $\sigma_i,\ i=1,..,l$. Let $H:=\sO_{\sX}(1)$ be an ample line bundle coming from the projective normal coarse moduli space $X$. Let $\Xi$ be the generating sheaf with fixed equivariant structure and let $P_{\Xi}$ denote the modified Hilbert polynomial. Let $M^s_{P_{\Xi}}$ denote the moduli space of modified Gieseker stable sheaves with fixed modified Hilbert polynomial $P_{\Xi}$. This moduli space admits a natural $\mathbb{T}$-action induced by the torus action on the stack, and we denote the fixed point locus by $(M^s_{P_{\Xi}})^{\mathbb{T}}$.\par
Let $\vec\chi$ be a characteristic function that, corresponding to a torsion-free toric sheaf, encodes the dimensions of the weight spaces associated with the torus action and simultaneously fixes a modified Hilbert polynomial. \par
For a stable toric torsion free sheaf $\mathscr{F},$ we have an unique choice (Remark \ref{Stable remark}) of a box element $b_i\in B_{\sigma_i}(\mathbb{T})$ for every open chart $\mathscr{U}_{\sigma_i}.$ A box element on each open chart parametrizes the non-primitive, non-degenrate torus action on that open chart. Let $(b_1,..b_l)\in \prod_{i=1}^{l}B_{\sigma_{i}}(\mathbb{T}),$ be the unique $l-$tuple of box elements for $\mathscr{F}.$ We define the characteristic function as
$$_{(b_1,..,b_l)}\vec\chi_{\mathscr{F}}:(\mathbb{Z}^d)^{l}\to \mathbb{Z}^l$$ given by
$$\big(\ _{(b_1,..,b_l)}\chi^{\sigma_1}_{\mathscr{F}}(m_1),..,\ _{(b_1,..,b_l)}\chi^{\sigma_l}_{\mathscr{F}}(m_l)\ \big)$$
$$=\big(\mathsf{dim}_{\mathbb{C}}(_{b_1}F^{\sigma_{1}}(m_1)),..,\mathsf{dim}_{\mathbb{C}}(_{b_l}F^{\sigma_l}(m_l))\big).$$
Here $_{b_i}F^{\sigma_{i}}(m_i)$ represents the vector spaces corresponding to the $S-family$ of the toric subsheaf $_{b_i}\mathscr{F}|_{\mathscr{U}_{\sigma_i}}.$\par  
We denote the moduli space of stable torsion free toric sheaf with fixed characteristic function $\vec\chi$ by $M^s_{\vec\chi},$ constructed using GIT. Such moduli spaces are combinatorial in nature. Let $(\chi_{P_{\Xi}})^{fr}$ denote the set of framed characteristic functions that give rise to the above Hilbert polynomial. There are finitely many characteristic functions giving rise to a fixed modified Hilbert polynomial. Then the following isomorphism of finite type quasi-projective $\mathbb{C}-$ schemes holds (Theorem \ref{Theorem 4.5}). 
\begin{thm}
Let $\sX$ be a smooth projective toric DM stack. Let $\sO_{X}(1),\Xi$ be an ample line bundle and a generating sheaf with fixed equivariant structure, respectively. Let $P_{\Xi}$ be a choice of a modified Hilbert polynomial of degree $dim(\sX). $ Then there exists a canonical isomorphism,
$$(M^{s}_{P_{\Xi}})^{\mathbb T} \cong \coprod \limits_{\vec\chi\in (\chi_{P_{\Xi}})^{fr}}M^{s}_{\vec\chi}.$$
\end{thm} \par

This theorem provides a concrete description of the torus-fixed points in the moduli space by relating them to moduli spaces parametrizing sheaves with prescribed combinatorial data encoded by characteristic functions.

Following the same framework, we obtain a similar decomposition result for slope-stable toric reflexive sheaves on $\sX$. Reflexive sheaves form an important subclass of torsion-free sheaves, often arising in contexts such as the study of instantons and stable bundles. Now, let $P_{\Xi}$ be the choice of a Hilbert polynomial for a reflexive sheaf and denote the moduli space of modified slope-stable reflexive sheaves by $M^{\mu s}_{P_{\Xi}}$. We denote the set of framed characteristic functions for toric reflexive sheaves giving rise to the above choice of Hilbert polynomial given by $(\chi^r_{P_{\Xi}})^{fr}$. We denote by $M^{\mu s}_{\vec\chi}$ the moduli space of the stable toric reflexive sheaves with characteristic function $\vec\chi.$ Then the following holds (Theorem \ref{Final2}).
\begin{thm}
Let $\sX$ be a smooth projective toric DM stack with $X$ as its projective normal coarse moduli space, an ample line bundle $H$ and a generating sheaf $\Xi$ with a fixed equivariant structure. Let $P_{\Xi}$ be a choice of a modified Hilbert polynomial of a reflexive sheaf. Then we have a canonical isomorphism of quasi-projective $\mathbb C$ schemes between the fixed point locus and the disjoint union of properly GIT stable moduli spaces with framed characteristic functions corresponding to the Hilbert polynomial given as,
$$(M^{\mu s}_{P_{\Xi}})^{\mathbb T} \cong \coprod \limits_{\vec\chi\in (\chi^r_{P_{\Xi}})^{fr}}M^{\mu s}_{\vec\chi}.$$
\end{thm}


In this paper, we generalize the work of \cite{Kly}, \cite{Per}, \cite{Koo1}, \cite{Koo2}, who study the moduli of pure sheaves with supports on smooth toric varieties. In the smooth toric variety setting, the natural action of the algebraic torus lifts to the moduli space of Gieseker stable sheaves, and the explicit description of the fixed point locus with a fixed Hilbert polynomial can be obtained in terms of properly GIT stable moduli of sheaves with fixed characteristic functions. These descriptions are combinatorial in nature, and such a decomposition of the fixed point locus is achieved by matching Gieseker stability with GIT stability—a key technical insight developed in \cite{Koo1}, \cite{Koo2}. In \cite{Koo3}, the above-mentioned techniques have been successfully used to understand and compute topological properties, such as Betti numbers and Euler characteristics, of moduli spaces of stable sheaves on smooth toric varieties.

Extending these ideas to the setting of smooth projective toric DM stacks, we generalize the techniques of \cite{GJK}, where the authors apply similar methods to study moduli of torsion-free sheaves on weighted projective planes. Subsequently, this work was extended to the case of stacky Hirzebruch orbifolds by \cite{WW}, providing important evidence that such decomposition results hold in a broader class of stacks.

Here, we study the moduli of torsion-free sheaves on an arbitrary smooth projective toric DM stack following the techniques developed in the aforementioned works. A key step in our approach is a gluing formula for torsion-free toric sheaves. We describe sheaves on the open substacks $\mathscr{U_{\sigma}}\cong [{\mathbb{C}^d}/{G_{\sigma}}] $ corresponding to top cones as bi-graded modules with respect to the $X(\mathbb T)$ and $X(G_{\sigma})$ gradings. This local data yields vector spaces equipped with fine gradings, referred to as stacky $S$-families. We then formulate a Gluing condition that specifies when these local data patch together to define a global toric torsion-free sheaf. This condition is derived by analyzing the intersection of two open substacks associated with distinct top cones, pulling back the local data under étale morphisms, and imposing compatibility equalities. The general Gluing formula presents significant challenges. Applying the Gluing formula we provide conditions in examples where a certain higher rank toric reflexive sheaf decomposes equivariantly into direct sum of toric reflexive sheaves of smaller rank. Moreover, we also show that the characteristic function corresponding to a modified stable toric torsion free sheaf has one box summand per chart. This enables us to match the GIT stability to the Gieseker/slope stability.

Following \cite{Koo2} and \cite{WW}, we obtain combinatorial descriptions of characteristic functions and construct moduli spaces of properly GIT stable sheaves with fixed characteristic functions. Adopting the same recipe as in \cite{Dol}, we construct equivariant line bundles that establish a correspondence between GIT stability and modified Gieseker stability. At this stage, we rely on \cite{Nir}, where moduli spaces of Gieseker semistable sheaves on smooth projective DM stacks are constructed by defining the modified Hilbert polynomial.

In the final section, we verify that the torus action lifts naturally to the moduli space of modified Gieseker stable sheaves. Using deformation-theoretic arguments analogous to those in the earlier references, we establish the desired decomposition result for the fixed point locus. We conclude the paper by working out explicit examples of the generating functions of rank 1 torsion free modified slope-stable sheaves for smooth toric surface DM stacks.

The work of \cite{GJK} presents interesting generalizations of the computations in \cite{Koo3}, extending them to more singular ambient spaces. The results of \cite{Got} and \cite{LGot} have been extended to the stacky Hirzebruch orbifold case for ranks 1 and 2 in \cite{WW}. Extending such investigations to higher-dimensional stacky toric threefolds remains an open and promising direction. Additionally, similar studies on weighted blow-ups of toric stacks may yield new insights into wall-crossing phenomena in moduli spaces. Moreover, the explicit description of the fixed point locus is a powerful tool for computing invariants such as Euler characteristics via torus localization techniques. We plan to address these questions in future work.

\subsection{Outline}
Here is the short outline of the paper.  In \S \ref{sec_preliminaries} we review the modified slope stability for torsion free sheaves on a smooth projective Deligne-Mumford stack $\sX.$ Next we recall the smooth toric DM stacks as in \cite{BCS}. Describing $S$-families we describe the toric torsion free sheaves and prove a Gluing formula (\ref{Gluing Formula}). We define characteristic functions \S \ref{Char} next and achieve a decomposition (\ref{Theorem 4.5}) as in (\cite{Koo2},Cor 4.10) after matching GIT and modified Gieseker stability \S \ref{GIT and Gieseker}. 


\subsection*{Acknowledgments}

I would like to thank Yunfeng Jiang for his constant help, discussions, and guidance while preparing this manuscript. I have benefited from conversations with Amin Gholampour, Dan Edidin, Jayan Mukherjee, Yohsuke Imagi, Daniel Skodlerack, and Ziyu Zhang regarding the details of this paper.


\section{Preliminaries on modified stability}\label{sec_preliminaries}

In this section, we review the modified semi-stability for Deligne-Mumford stacks, recalling the necessary definitions.

\subsection{Notations}

We fix some notation for a smooth projective  Deligne-Mumford stack $\sX$ of dimension $d$ with coarse moduli space $X,$ denoted by $\pi:\sX \to X$. This section necessarily works out for any tame DM stack, but in characteristic 0 we have the tame condition irrespective.

Let $\sI$ be the index set of the components of the inertia stack 
$I\sX$ such that
$$I\sX=\bigsqcup_{g\in\sI}\sX_g.$$
The index set $\sI$ consists of conjugacy classes $(g)$ of the local stacky group $G$ of $\sX$. 
We always use $\sX_0=\sX$ to represent the trivial component.  For example, if $\sX=[Z/G]$ is a global  quotient stack, where $Z$ is a quasi-projective scheme and $G$ is a diagonalizable group scheme acting linearly on $Z$, then 
$I\sX=\bigsqcup_{(g)}[Z^g/C(g)]$ where $(g)\in G$ represent the conjugacy classes of the geometric points in $G$ that fix geometric points in $Z$ denoted by $Z^g$ and $C(g)$ represent the centralizer of $g\in G$. Any component $\sX_g\subset \sX$ in the inertia stack $I\sX$ is a closed substack of $\sX$ that is isomorphic to $[Z^g/C(g)]$. $Z^g$ embeds in $Z$ which is equivariant under the action of $C(g)$ (as a subgroup of $G$)  and $G,$ respectively which explains the component of inertia being a closed substack.
For all the discussions regarding Orbifold Riemann Roch we cite \cite[\S 3.3]{Nir}  or \cite{Toe} and \cite{Tse} as references.
We denote by $I\sX_1\subset I\sX$  the substack of $I\sX$ consisting of components $\sX_g$ such that their codimension in $\sX$ is one. 
Let $\pr: I\sX\to \sX$ be the map from the inertia stack $I\sX$ to $\sX$. 

For $\sX$, we write 
$$H_{\CR}^*(\sX)=H^*(I\sX)=\bigoplus_{g\in\sI}H^*(\sX_g)$$
to be the Chen-Ruan cohomology with $\qq$-coefficients. For any torsion free coherent sheaf $E$ on $\sX$, we use $c_i(E)$ to represent the Chern classes of $E$ on $\sX$, and $c_i(E)\in H^{2i}_{\CR}(\sX)$. 

In the component $\sX_g\subset I\sX$, at a point $(x,g)\in \sX_g$, let 
$$T_{x}\sX=\bigoplus_{0\leq f<1}\left(T_{x}\sX\right)_{g,f}$$
be the eigenspace decomposition of $T_{x}\sX$ with respect to the stabilizer action and $g$ acts on $\left(T_{x}\sX\right)_{g,f}$ by $e^{2\pi i f}$ where $f=e/m,e,m\in \mathbb Z, 0\leq e<m,$ where $m$ is the order of the cyclic group generated by $<g>.$ See (\cite{Nir},Section 3.3) for more details.

Let $E\in\Coh(\sX)$ be a coherent sheaf on $\sX$, we have an eigenbundle decomposition of $\pr^*E $ and on $\pr^*E|_{\sX_g}$ we have
$$\pr^*E|_{\sX_g}=\bigoplus_{0\leq f<1}(\pr^*E)_{g,f}$$
with respect to the action of the stabilizer of $\sX_g$, where the element $g$ acts on $(\pr^*E)_{g,f}$ by $e^{2\pi i f}$.  
Then the orbifold Chern character is:
\begin{equation}\label{eqn_chern}
\widetilde{\Ch}(E)=\bigoplus_{g\in \sI}\sum_{0\leq f<1}e^{2\pi i f}\Ch((\pr^*E)_{g,f}),
\end{equation}
where $\Ch$ is the general Chern character.  Let $l_{g,f}$ be the rank of $(\pr^*E)_{g,f}$.   The orbifold Todd class of $T\sX$ is given by 
\begin{equation}\label{Todd}
\widetilde{\Td}(T\sX)=\bigoplus_{g\in \sI}\prod_{\substack{0\leq f<1\\
1\leq i\leq r_{g,f}}}\frac{1}{1-e^{-2\pi i f}e^{-x_{g,f,i}}}\prod_{f=0}\frac{x_{g,0,i}}{1-e^{-x_{g,0,i}}},
\end{equation}
where $(\pr^*T\sX)_{g,f}$ has rank $r_{g,f}$ and $x_{g,f,i}$ are Chern roots. 

For any coherent sheaf $E$ on $\sX$, the orbifold Riemann-Roch theorem \cite[Theorem 5.4]{Toe} gives:
\begin{equation}\label{eqn_orbifold_RR}
\chi(\sX, E)=\int_{I\sX}\widetilde{\Ch}(E)\cdot \widetilde{\Td}(T\sX).
\end{equation}

\subsection{Modified stability}

Let $\sX$ be a smooth tame projective Deligne-Mumford stack of dimension $d$.  We  choose the polarization $\sO_X(1)$ on its coarse moduli space $\pi: \sX\to X$.
Let $H:=c_1(\sO_X(1))$.
Recall from \cite[\S 2]{Nir}, 

\begin{defn}\label{defn_generating_sheaf}
A locally free sheaf $\Xi$ on $\sX$ is $\pi$-very ample if  for every geometric point of $\sX$ the representation of the stabilizer group at that point contains every irreducible representation of the stabilizer group.  We call $\Xi$ a generating sheaf. Such a vector bundle always exists for smooth quotient Deligne-Mumford stacks, see Theorem 5.7 in \cite{OS}.
\end{defn}

Let $\Xi$ be a locally free (generating) sheaf  on $\sX$. We define a functor
$$F_{\Xi}: D\Coh_{\sX}\to D\Coh_{X}$$ by
$$F\mapsto \pi_*\cHom_{\sO_\sX}(\Xi, F)$$
and a functor 
$$G_{\Xi}: D\Coh_{X}\to D\Coh_{\sX}$$ by
$$F\mapsto \pi^*F\otimes \Xi.$$
From \cite[\S 5]{OS}, the functor $F_{\Xi}$ is exact since the dual $\Xi^{\vee}$ is locally free and the pushforward $\pi_*$ is exact.  The functor $G_{\Xi}$ is not exact unless $\pi$ is flat. For example, if $\pi$ is a flat gerbe or a root stack, it is flat.

Fix a generating sheaf $\Xi$ on $\sX$. We call the pair $(\Xi, \sO_X(1))$ a polarization of $\sX$. 
Let $E$ be a coherent sheaf on $\sX$, we define the support of $E$ to be the closed substack associated with the ideal
$$0\to \sI\to \sO_{\sX}\to \cEnd_{\sO_{\sX}}(E).$$
So $\dim(\Supp E)$ is the dimension of the substack associated with the ideal $\sI\subset \sO_{\sX}$ since $\sX$ is a Deligne-Mumford stack. 
A pure sheaf of dimension $d$ is a coherent sheaf $E$ such that for every non-zero subsheaf $E^\prime$ the support of $E^\prime$ is of pure dimension $d$. 
For any coherent sheaf $E$, we have the torsion filtration:
$$0\subset T_0(E)\subset \cdots \subset T_l(E)=E$$
where every $T_i(E)/T_{i-1}(E)$ is pure of dimension $i$ or zero, see \cite[\S 1.1.4]{HL}.

\begin{defn}\label{defn_Gieseker}
The modified Hilbert polynomial $H_{\Xi}$ of a coherent sheaf $E$ on $\sX$ is defined as:
$$H_{\Xi}(E,m)=\chi(\sX, E\otimes\Xi^{\vee}\otimes \pi^*\sO_{X}(m))
=H(F_{\Xi}(E)(m))=\chi(X, F_{\Xi}(E)(m))$$ where $H$ is the usual Hilbert polynomial defined on the coarse moduli space.
\end{defn}

Let $E$ be of dimension $d$, then we can write:
$$H_{\Xi}(E,m)=\sum_{i=0}^{d}\alpha_{\Xi, i}(E)\frac{m^i}{i!}$$
which is induced by the case of schemes. The modified Hilbert polynomial is additive on short exact sequences since the functor $F_{\Xi}$ is exact. If we do not choose the generating sheaf $\Xi$, the Hilbert polynomial $H$ on $\sX$ will be the same as the Hilbert polynomial on the coarse moduli space $X$.  
The {\em reduced modified Hilbert polynomial} for the pure sheaf $E$ is defined as 
$$h_{\Xi}(E)=\frac{H_{\Xi}(E)}{\alpha_{\Xi, d}(E)}.$$
Let $E$ be a pure coherent sheaf.  We call $E$ Gieseker semistable if for every proper subsheaf $E^\prime\subset E$, 
$$h_{\Xi}(E^\prime)\leq h_{\Xi}(E).$$
We call $E$ stable if $\leq$ is replaced by $<$ in the above inequality.


\begin{defn}(\cite{Nir})\label{defn_modified_slope_stability}
 We define the modified slope of  $E$ by 
 $$\mu_{\Xi}(E)=\frac{\alpha_{\Xi, l-1}(E)}{\alpha_{\Xi, l}(E)}.$$
 Then $E$ is
modified slope (semi)stable if, for every proper subsheaf $F\subset E$, 
$$\mu_{\Xi}(F)(\leq) < \mu_{\Xi}(E).$$
\end{defn}
The notion of $\mu$-stability and semistability is related to the Gieseker stability and semistability in the same way as schemes, i.e.,
$$\mu-\text{stable}\Rightarrow \text{Gieseker stable}\Rightarrow \text{Gieseker semistable}\Rightarrow \mu-\text{semistable}$$

\begin{rmk}\label{rmk_stability_quotient}
Recall that  $I\sX_1\subset I\sX$ is the substack of $I\sX$ consisting of components such that the codimension of $\sX_g\subset \sX$  is one. 
Suppose our Deligne-Mumford stack $\sX$ is a global quotient stack, which means $\sX=[Z/G]$ where $Z$ is a quasi-projective scheme and $G$ is a diagonalizable group scheme. Assume that we can choose the generating sheaf $\Xi$ on $\sX$ such that its restriction on any component in $I\sX_1$ is a sum of  locally free sheaves of  the same rank. If the sheaf $E$ has dimension $d$, then    \cite[Proposition 3.18]{Nir} shows that 
$$\deg_{\Xi}(E)=\frac{1}{\rk(\Xi)}\alpha_{\Xi, d-1}(E)-\frac{\rk(E)}{\rk(\Xi)}\alpha_{\Xi,d-1}(\sO_{\sX}).$$
Here $\alpha_{\Xi, d-1}(E)=\int_{\sX}c_1(E)\cdot H$. 
\end{rmk}

We note that in \cite[Proposition 2.3]{Nir} such a generating sheaf exists in the case of toric orbifold DM stacks.
\section{Toric Sheaves on Toric DM stacks}
\subsection{Toric DM stacks}

In this section 
we always fix $\sX$ as a smooth toric projective Deligne-Mumford stack of dimension $d$. We always fix a generating sheaf with a fixed torus equivariant structure $\Xi$ on $\sX.$ We talk about modified semi-stable sheaves in characteristic $0$ with respect to  $\Xi,\ \pi^*H$ where $H$ is an ample line bundle on $X,$ which is a projective normal toric variety as the coarse moduli space and $\pi:\sX\to X$ is the coarse moduli map. For the rest of the paper we always work with this polarization data $
 (\Xi,\ \pi^*H).$ \par 
Most of the materials for this section are from \cite{BCS},\cite{FMN} and \cite{JT}. We recall the necessary things about smooth toric DM stacks.\par
Let $$\beta:\mathbb{Z}^n \to N$$ be a map of abelian groups with finite cokernel and $N$ is finitely generated of rank $d.$ We write $\overline{N}$ to denote the lattice generated by $N$ in $N_{\mathbb Q}.$ Let $\Sigma$ be a rational simplicial fan in $N_{\mathbb Q}$; every cone is generated by linearly independent vectors. Let $\rho_{i},i=1,,n$ be the rays in $\Sigma$ spanning $N_{\mathbb Q}$ and let integers $b_i$ generate $\rho_i.$ Denote the map $\beta$ by the images $(b_1,..,b_n).$ The triple $(N,\Sigma,\beta)$ is our stacky fan.\par
Corresponding to the map $\beta,$ we consider the following exact sequence obtained from applying Gale Duality,
$$0\to N^{*}\to \mathbb{Z}^n \xrightarrow{\beta^{\vee}} H^{1}(Cone(\beta)^{*})\to Ext^{1}(N,\mathbb{Z})\to 0.$$ Taking dual with $\mathbb{C}^{*}$ we obtain 
$$1\to \mathscr{K} \to G \xrightarrow{\alpha} (\mathbb{C^{*}})^{n}\to \mathbb{T} \to 1.$$ 
We explain the notation as follows.\\$\ \mathscr{K}=Hom(Ext^{1}(N,\mathbb{Z}),\mathbb{C}^{*}),\ G= Hom(H^{1}(Cone(\beta)^{*}),\mathbb{C}^{*}).$
We also refer as, $DG(\beta):=H^{1}(Cone(\beta)^*).$\par
The stacky fan comes with a group action on a quasi affine variety $Z.$ This action is described by using $\mathbb{C}[x_1,..,x_n]$ as the affine co-ordinate ring of $\mathbb{C}^n$ and we consider the irrelevant ideal given by, $J_{\Sigma}=<\prod_{\rho\nsubseteq \sigma} x_i :\sigma \in \Sigma >$ and $Z:=\mathbb{C}^n-V(J_{\Sigma}).$ The $\mathbb{C}$-valued points of $Z$ are given by $z\in \mathbb{C}^n$ such that the cone generated by $\rho_i:z_i=0$ belongs to $\Sigma.$ The natural multiplication action on $(\mathbb{C^*})^n$ extends to $G$ action on the affine variety via the map $\alpha$ and $Z$ being $G$-invariant we have an action of $G$ on $Z$ and we have $\sX(\Sigma):=[Z/G]$ the quotient stack associated to the groupoid
$s,t:Z\times G \to Z,s,t$ being projection on first factor and action respectively.\par 
In \cite{FMN} they show that the above definition of smooth toric DM stack as a stacky fan can also be realized as the lifting of the torus action from itself onto the stack as in the case of smooth toric varieties. This equivalence of ideas is what we exploit next to describe coherent sheaf on such stacks. \par 
We have the following classification of smooth toric DM stacks in \cite{FMN}. If $N$ is free, the corresponding stack is a toric orbifold and the torus is given by $(\mathbb{C^*})^d.$ If $N$ has torsion and $\beta$ generates the torsion of $N$ then the quotient stack is a $\mathscr{K}$-gerbe over its rigidification \cite{JT},\cite{FMN}. If $\beta$ does not generate the torsion part, the stack is still a $\mathscr{K}'$-gerbe over its rigidification where $K'$ is a subgroup pf $\mathscr{K}.$\par 
$\sX$ that admits an open cover by substacks of the form $\mathscr{U}_{\sigma}\cong[\mathbb{C}^d/G_{\sigma}]$ when we consider the open substacks corresponding to the top cones. In the next section, we will exploit the nature of such open covers, describe torsion free coherent sheaves upon them, and describe the sheaves through Gluing.
\subsection{Toric sheaves on Toric DM stacks}
In this section, we will briefly recall (\cite{GJK}, Section 3). Let $\mathbb{T}$ act linearly and non degenerately on $\mathbb{C}^d$ ( this is the action for the torus on smooth toric DM stacks ) and let the action be given by $$\lambda.x_i=\chi(m_i)(\lambda).x_i,$$ where the choice of co-ordinates is unique up to scaling and re-ordering of $x_i,\ m_i\in X(\mathbb{T})$ and $\chi(m_i)$ denotes the function as a character.
\begin{defn}
The box associated to such an action as above $B_{\mathbb{T}}\subset X(\mathbb{T})$ (the character group of the torus), is the collection of all elements of the form $\sum_{i=1}^{d}q_im_i$ with rational $0\leq q_1,..,q_d<1$ \label{Box}.    
\end{defn}
For a fixed top cone $\sigma,$ and for the torus action on $\mathscr{U}_{\sigma}$ we denote the box of the action by $B_{\sigma}(\mathbb{T}).$  In particular, for any non-degenerate integer matrix $M_{d\times d}$ we can consider the box given by its columns. In our case, this matrix will be given by the torus action on affine spaces given by the characters of the torus, and hence given by columns of integers associated with the character. The size of the box for such a matrix is given by the determinant of the matrix and is equal to the size of the co-kernel of the map defined by the matrix. We deal with non primitive actions in this paper, which means the box for such action is non-trivial. \par 
Following the above notation, we will refer to a sheaf $\mathscr{F}$ as a torus equivariant or toric sheaf on $\mathbb{C}^d.$ Given such an equivariant sheaf, let
$$H^{0}(\mathbb{C}^d,\mathscr{F})=\bigoplus_{m\in X(\mathbb{T})}F(m)$$ be the corresponding graded $X(\mathbb{T})$ module. Multiplication by $x_i$ gives maps
$$\chi_i(m):F(m)\to F(m+m_i),$$ for all $m\in X(\mathbb{T})$ that further satisfy
$$\chi_{j}(m+m_i)\circ\chi_{i}(m)=\chi_{i}(m+m_j)\circ\chi_{j}(m)$$ for all $m\in X(\mathbb{T}),i,j=1,,d.$
\begin{defn}
Let $X(\mathbb{T})$ be the character group of an algebraic torus $\mathbb{T}$ of dimension $d.$ Suppose that we are given $\mathbb{Z}$-independent elements $m_1,..,m_d \in X(\mathbb{T}).$ An $S- family\ \hat{F}$ consists of the following: a collection of complex vector spaces $\{F(m)\}_{m\in X(\mathbb{T})}$ and linear maps $\{\chi_i(m)\}_{m\in X(\mathbb{T}),i=1,,d}$ satisfying the above compatibility described.
A morphism of $S-families\ \hat{F},\hat{G}$ is a family $\hat{\phi}$ of linear maps
$$\{\phi(m):F(m)\to G(m)\}_{m\in X(\mathbb{T})}$$ commuting with the $\chi_{i}(m).$
\end{defn}
Next we define the box decomposition of a toric sheaf on the open chart $\mathscr{U}_{\sigma}.$
\begin{defn}
For an $S-family\ \hat{F}$ we define $_{b}\hat{F}$ to be the $S-subfamily$ where $b\in B_{\sigma}(\mathbb{T})$ consisting of all vector spaces $F(m+b):=_{b}F(m)$ where $m=\sum_{i=1}^{d}l_{i}m_{i}.$
\end{defn}
Then (Prop. 3.4 in \cite{GJK}) is in force. We recall it here. Denote the toric subsheaf by $_{b}\mathscr{F}$ corresponding to the $S-subfamily,\ _{b}\hat{F}.$ Then the following holds $$\mathscr{F}\cong{\bigoplus _{b\ \in B_{\sigma}(\mathbb{T})}}{_{b}\mathscr{F}}.$$
We refer to \cite{GJK} for the $S$-family description of coherent, torsion free toric sheaves, reflexive sheaves and locally-free sheaves.\par
Next we recall the definition of stacky $S-$ families as sheaves on $[\mathbb{C}^d/H]$ where the action of $H$ a finite abelian group commutes with the torus action.\par
A torus equivariant sheaf $\mathscr{F}$ on $[\mathbb{C}^d/H]$ is equivalent to $H^0(\mathbb{C}^d,\mathscr{O}_{\mathbb{C}^d})$ modules with $X(\mathbb{T})$ grading and $X(H)$ fine-grading.\par 
This description in turn gives us a S-family $\hat{F}$ which for each $m\in X(\mathbb{T})$ is given by,
$$F(m)=\bigoplus_{n\in X(H)}F(m)_{n}.$$ Under the declared co-ordinates since the torus and the $H$ actions commute we further have,
$$h.x_i=\chi(n_i)(h).x_i,$$ for a unique $n_i\in X(H).$
\begin{defn}
A  $stacky\ S-family$ consists of the following: a collection of vector spaces $\{F(m)_n\}_{m\in X(\mathbb{T}),n\in X(H)},$ a collection of linear maps $$\{\chi_{i}(m):F(m)\to F(m+m_i)\}_{i=1,,d,m\in X(\mathbb{T})}$$  satisfying
$$\chi_{i}(m):F(m)_n\to F(m+m_i)_{n+n_i},$$
$$\chi_{j}(m+m_i)\circ\chi_i(m)=\chi_{i}(m+m_j)\circ \chi_{j}(m),$$ for all $i,j=1,,d,m\in X(\mathbb{T}),n\in X(H).$
\end{defn}
The morphisms are given by the morphisms of $S-families $ which respect the fine-grading. The obvious notions of coherent, torsion free, reflexive sheaves as $stacky\ S-families$ generalize in this setting where the underlying $S-family$ has a fine-grading and the morphisms respect the fine-grading as well.

\subsection{Gluing of toric torsion free sheaves}

Let $\mathscr{F}$ be a torsion free coherent sheaf on a Toric DM stack $\sX$ given by $(N,\Sigma,\beta).$ Denote by $\mathscr{U}_1\cong [Z_1/G_{\sigma_1}]$ the open substack of dimension $d$ given by the top dimensional cone $\sigma_1.$ Consider the fibered product, $\pi_i:\mathscr{U}_1\times_{\sX} \mathscr{U}_2\longrightarrow \mathscr{U}_i,i=1,2.$ Consider two top cones $\sigma_1$ and $\sigma_2$ and consider $\alpha: G\rightarrow \mathbb{C^*}^{n}$ corresponding to $\alpha_i ,1 \leq i\leq n \in X(G),$ obtained via Gale Duality. Note that $G\cong \mathbb{C^*}^{n-d}\times \mu_1..\times \mu_{r},$ where $\mu_i$ is the $i-$th group of unity and $r\leq d$. Let us denote the co-ordinates of the ambient $\mathbb{C}^{n}$ and hence the rays by $$(x_1,.,x_n)$$ corresponding to $$(\rho_{1},.,\rho_{n})$$ and let the top cones be generated by the following rays indexed by the ordered tuple $$(x_{\sigma_{1}(1)},.,x_{\sigma_{1}(d)})$$ corresponding to the ordered tuple $$(\sigma_{1}(1),.,\sigma_{1}(d))$$ for $\sigma_1$ and $$(x_{\sigma_{2}(1)},.,x_{\sigma_{2}(d)})$$ corresponding to the ordered tuple $$(\sigma_{2}(1),.,\sigma_{2}(d))$$ for $\sigma_2.$\newline Also note that $$\sigma_{1}(i)=\sigma_{2}(i)\ \forall i=1,.,d-p.$$
$$\sigma_{1}(i)\neq \sigma_{2}(i),\ \forall i=d-p+1,.,d.$$\newline 
Then one obtains $$\mathscr{U}_1\times_{\sX} \mathscr{U}_2 \cong 
[\frac{\mathbb{C}^{d-p}\times G_{\sigma_1\cup\sigma_2}}{G_{\sigma_1} \times G_{\sigma_2}}],
$$
where $$G_{\sigma_1\cup\sigma_2}=\{\lambda\in G|\alpha_i(\lambda)=1 \forall \rho_i \notin \sigma_1 \cup \sigma_2 \}$$ where $\pi_1$ corresponds to $$\pi_1:\mathbb{C}^{d-p}\times G_{\sigma_1\cup\sigma_2} \rightarrow \mathbb{C}^{d-p}\times \mathbb{C^*}^{p}\rightarrow \mathbb{C}^d$$ and is further given by,
$$(\tau_1,..,\tau_{d-p},\lambda)\rightarrow \big((\alpha_{\sigma_{1}(1)}(\lambda)\tau_1,..,\alpha_{\sigma_{1}(d-p)}(\lambda)\tau_{d-p}),\alpha_{\sigma_1(d-p+1)}(\lambda),.,\alpha_{\sigma_{1}(d)}(\lambda)\big).$$\newline Similarly $\pi_2$ corresponds to $$\pi_2:\mathbb{C}^{d-p}\times G_{\sigma_1\cup\sigma_2} \rightarrow \mathbb{C}^{d-p}\times \mathbb{C^*}^{p}\rightarrow\mathbb{C}^d$$ further given by,
$$(\tau_1,,,\tau_{d-p} ,\lambda)\rightarrow \big((\tau_1,..,\tau_{d-p}),\alpha_{\sigma_2(d-p+1)}(\lambda)^{-1},.,\alpha_{\sigma_{2}(d)}(\lambda)^{-1}\big).$$ The morphisms on the groups are obtained from projections on each factor.\par 
Before further simplifications we motivate ourself with this example which tells us how to reduce $\mathscr{U}_1 \times_{\sX} \mathscr{U}_2$.
\begin{example} (\cite{FMN},Ex. 7.31)
Let $\beta:\mathbb{Z}^2\to \mathbb{Z}$ be given  by $a>0,b>0$ and consider the rays in $\mathbb{R}$ given by the rays $(a,-b).$ The corresponding quotient stack is given by $$\mathbb{C}^*\times \mu_d\times (\mathbb{C}^2- \{0\})\to (\mathbb{C}^2-\{0\})$$
$$((\lambda,t),(x_1,x_2))\to (\lambda^{m/a}t^{k_2}x_1,\lambda^{m/b}t^{-k_1}x_2)$$ where,
$$d=gcd(a,b),m=lcm(a,b)$$ and $k_1,k_2$ satisfies
$$bk_1+ak_2=d.$$
In this above example, $\mathscr{U}_1\cong [\mathbb{C}/\mu_a]$ and $\mathscr{U}_2\cong[\mathbb{C}/\mu_b]$ and $\mathscr{U}_1\times_{\sX} \mathscr{U}_2\cong [\mathbb{C}^*\times \mu_d/(\mu_a\times \mu_b)].$ The actions can be derived using the construction of \cite{BCS} and then we can simplify the above by observing that $\mathbb{T}\cong\mathbb{C}^*\cong [\mathbb{C}^*\times \mu_d/(\mu_a\times \mu_b)]\cong[\mathbb{C}^*/\mu_m].$ In the following simplification we generalize this example for two top cones and their intersection.
\end{example}

We further simplify this representation. Note that there is a morphism from $({\mathbb{Z}^d})^{*}$ to $({\mathbb{Z}^{d+p}})^{*}$ given by the union of the two top cones namely,
$$\beta_{\sigma_1\cup\sigma_2}:\mathbb{Z}^{d+p} \to N $$and observe that the Gale Dual of this map is $DG_{\sigma_1\cup\sigma_2}.$ As an application of the naturality of Gale Duality one observes that $DG_{\sigma_1\cup\sigma_2}$ admits a surjection from $DG(\beta)$ and dualizing by $\mathbb{C^*}$ one obtains the group $G_{\sigma_1\cup\sigma_2}$ embedded in $G.$ Further note that via Gale Duality, $DG_{\sigma_1\cup\sigma_2}$ surjects onto $DG_{\sigma_1}$ and $DG_{\sigma_2}$ via the maps say $\phi_1,\phi_2$ respectively. Denote the natural induced map as,$$(\phi_1,\phi_2):DG_{\sigma_1\cup\sigma_2}\to DG_{\sigma_1}\bigoplus DG_{\sigma_2}.$$
Denote by $DH:=(DG_{\sigma_1\cup \sigma_2})_{tor}$ the torsion subgroup of $DG_{\sigma_1\cup \sigma_2}.$\par 
Now we prove a lemma which will be helpful in carrying out the reduction,
\begin{lem}\label{Cokernel}
The surjection $\phi_i$ restricted to $DH:=(DG_{\sigma_1\cup \sigma_2})_{tor}$ is an injection into each summand $DG_{\sigma_i}$ for $i=1,2.$   
\end{lem}
\begin{proof}
We prove the statement for $i=1.$ From the above discussion,
$$\beta_{\sigma_1\cup\sigma_2}:\mathbb{Z}^{d+p}\to N$$ and using Gale duality we arrive at the following diagram. Denote $DG_{\sigma_1\cup\sigma_2}\cong DH\bigoplus \mathbb{Z}^p$ and we use (\cite{BCS},section 2) as a reference. One has 
\begin{center}
~~~~~~\xymatrix{
  0 \ar[r] & (\mathbb{Z}^{d+r})^{*} \ar[d]_-{\id} \ar[r]^-{[BQ]^{*}} & (\mathbb{Z}^{d+r+p})^{*} \ar[d]_-{i^*} \ar[r]^-{q} & (DH \bigoplus(\mathbb{Z}^{p})^{*}) \ar[d]^-{i^*} \ar[r] & 0 \\
  0 \ar[r] & (\mathbb{Z}^{d+r})^{*} \ar[r]_-{[B_{\sigma_1}Q]^{*}} & (\mathbb{Z}^{d+r})^{*} \ar[r]_-{q} & DG_{\sigma_1} \ar[r] & 0
}
\end{center}
where the first $i^*$ denotes the dual to $i:\mathbb{Z}^d\to \mathbb{Z}^{d+p}$ corresponding to inclusion in first $d$ co-ordinates and the second $i^*$ denotes the induced map on the quotients.\par The map $\beta_{\sigma_1\cup\sigma_2}:\mathbb{Z}^{d+p}\to N$ lifts to the matrix $[B]:\mathbb{Z}^{d+p}\to \mathbb{Z}^{d+r}$ and $N$ being finitely generated abelian group of rank $d$ we have the following projective resolution of $N$ given by the matrix $[Q]$ as,
$$0\to \mathbb{Z}^{r}\xrightarrow{[Q]} \mathbb{Z}^{d+r}\to 0.$$ We denote by $$[B_{\sigma_1}]:\mathbb{Z}^{d}\to \mathbb{Z}^{d}$$ the matrix corresponding to the columns of $\sigma_1.$ 
Next we employ a change of co-ordinates for the matrices $[BQ],[B_{\sigma_1}Q]$ using the Smith-Normal form and obtain the following description given by,
\begin{center}
~~~~~~\xymatrix{
  0 \ar[r] & (\mathbb{Z}^{d+r})^{*} \ar[d]_-{[C]} \ar[r]^-{[\Lambda]} & (\mathbb{Z}^{d+r+p})^{*} \ar[d]_-{[B'Q']} \ar[r]^-{q} & (DH \bigoplus(\mathbb{Z}^{p})^{*}) \ar[d]^-{[\overline{B'Q'}]} \ar[r] & 0 \\
  0 \ar[r] & (\mathbb{Z}^{d+r})^{*} \ar[r]_-{[A]} & (\mathbb{Z}^{d+r})^{*} \ar[r]_-{q} & DG_{\sigma_1} \ar[r] & 0
}
\end{center}
where we represent the $i$-th row and $j$th column of the matrix as follows,
$[C]$ is an invertible matrix, $[\Lambda],[A]$ being injective diagonal non-singular matrices.

$[B'Q']$ represents $(d+r)\times (d+r+p)$ matrix where $[B']$ is the $(d+r)\times (d+r)$ block and $q$ denotes the corresponding quotient maps. \par Note that $[BQ]\circ [\Lambda]:(\mathbb{Z}^{d+r})^{*}\to (\mathbb{Z}^{d+r})^{*}$ is an injective matrix. This forces $[B']:(\mathbb{Z}^{d+r})^*\to (\mathbb{Z}^{d+r})^*$ to be injective. The following diagram is evident,
\begin{center}
~~~~~~\xymatrix{
  0 \ar[r] & (\mathbb{Z}^{d+r})^{*} \ar[d]_-{[C]} \ar[r]^-{[\Lambda]} & (\mathbb{Z}^{d+r})^{*} \ar[d]_-{[B']} \ar[r]^-{q} & DH  \ar[d]^-{[\overline{B'}]} \ar[r] & 0 \\
  0 \ar[r] & (\mathbb{Z}^{d+r})^{*} \ar[r]_-{[A]} & (\mathbb{Z}^{d+r})^{*} \ar[r]_-{q} & DG_{\sigma_1} \ar[r] & 0
}
\end{center}
We finish the proof by observing that $[\overline{B'Q'}]|_{DH}:DH\to DG_{\sigma_1}$ is given by $[\overline{B'}]$ and Snake lemma shows that the map is injective.
\end{proof}
Denote the cokernel of the restricted map $(\phi_1,\phi_2):DH\to DG_{\sigma_1}\bigoplus DG_{\sigma_2} $ as $X(K).$ We have the following exact sequence,
\begin{equation}\label{cokernel}
0\to DH\to DG_{\sigma_1}\bigoplus DG_{\sigma_2}\to X(K)\to 0.    
\end{equation}
We denote the image of $(l_1,l_2)\in DG_{\sigma_1}\bigoplus DG_{\sigma_2}$ under the surjection to $X(K)$ by $\overline{(l_1,l_2)}$.\par  Upon dualizing by $\mathbb{C}^{*}$ one obtains the following exact sequence,
\begin{equation}\label{eq:reduction}
1 \to K \to  G_{\sigma_1}\times G_{\sigma_2} \to H \to 1
\end{equation} \par

\begin{thm}
We have the following isomorphism of toric DM stacks given by,
$$\mathscr{U}_1\times_{\sX}\mathscr{U}_{2}\cong [\frac{\mathbb{C}^{d-p}\times \mathbb{C^*}^{p}}{K}].$$   
\end{thm}
\begin{proof}
The above exact sequence gives rise to these isomorphisms of toric DM stacks,
$$[\frac{\mathbb{C}^{d-p}\times \mathbb{C^*}^{p}\times H}{G_{\sigma_1}\times G_{\sigma_2}}]\cong [\frac{[\frac{\mathbb{C}^{d-p}\times \mathbb{C^*}^{p}\times H}{K}]}{H}]\cong [\frac{[\frac{\mathbb{C}^{d-p}\times \mathbb{C^*}^{p}}{K}]\times H}{H}]
\cong [\frac{\mathbb{C}^{d-p}\times \mathbb{C^*}^{p}}{K}].$$ 
The first isomorphism follows from the fact that $\mathbb{C}^{d-p}\times \mathbb{C^*}^{p}\times H$ is a $K$ torsor over ${[\frac{\mathbb{C}^{d-p}\times \mathbb{C^*}^{p}\times H}{K}]}$ and  $G_{\sigma_1}\times G_{\sigma_2}$ is a $K$ torsor over $H.$ Equivariant quotients induces the isomorphism.\par
The action of $K$ through the above exact sequence is trivial on $H$ which accounts for the next isomorphism.\par The action of $H$ on itself being free through multiplication and trivial on the other factor gives us the last isomorphism.

\end{proof}

Also we denote the morphisms $\pi_1$ and $\pi_2$ as the one induced by restricting $\mathbb{C^*}^{p}\times H$ to $\mathbb{C^*}^{p}$ and the group morphisms are induced likewise by restricting the projection map to $K.$
\par
Let $\mathscr{U}_1\times_{\sX}\mathscr{U}_2..\times_{\sX} \mathscr{U}_{k} \cong \mathbb{T}\cong (\mathbb{C^*})^d\times B\mu$ (note that $\sX$ is a $\mu$-gerbe and $B\mu$ is the classifying stack) be the torus obtained by intersecting $k$ top cones and the torus acts on itself by multiplication and the the action of the torus on each top cone and their fiber products extends the multiplication of the torus on itself because the projection maps are $\mathbb{T}-$ equivariant. Note that $\mathbb{C^*}^{p}$ is a diagonalizable algebraic subgroup of $G$ and is affine and denote the co-ordinates by $\lambda=(\lambda_1,..,\lambda_p).$\newline
Next we note that $[Z_1/G_{\sigma_1}]\cong [U_1/G]$ induced by the morphism $\phi_1:\mathbb{C}^d \cong Z_1\rightarrow U_1\cong \mathbb{C}^d\times \mathbb{C^*}^{n-d}\hookrightarrow Z,$ given by $$(x_1,..,x_d)\rightarrow (x_1,..,x_d,1_{\rho_i\notin \sigma_1}).$$ We also notice that $G$ acts on $U_1$ via \textit {Gale duality} and $g\in G$ belongs to $G_{\sigma_{1}}$ if and only if $g.Z_1 \cap Z_1$ is not empty. This implies $$g\in G\implies g\in G_1 \iff \lambda_i(g)_{\rho_i\notin \sigma_1}=1.$$\par Thus one obtains the action of $G_{\sigma_1}\hookrightarrow G$ on $\phi_1(Z_1)\cong Z_1$ induced by $G$ and using the description of the fibered product and equivariance under the action of $G_{\sigma_1}\hookrightarrow G$ and $G_{\sigma_1}\times G_{\sigma_2}$ on $Z_1$ and $\mathbb{C}^{d-p}\times \mathbb{C^*}^{p}\times H$ respectively for $\pi_1$ and similarly for $\pi_2$ we have the action of $G_{\sigma_1}\times G_{\sigma_2}$ specified on $\mathbb{C}^{d-p}\times \mathbb{C^*}^{p}\times H,$ which in turn induces the action of $K$ on $\mathbb{C}^{d-p}\times \mathbb{C^*}^{p}.$ So, let us denote the $\mathbb{T}$ action on $Z_1$ by $(m^1_1,..,m^1_d)\in \mathbb{Z}^d$ and by $(m_1^{2},m_2^{2},..,m_d^{2})\in \mathbb{Z}^d$ on $Z_2.$ Similarly let us denote by $(\chi_1,..,\chi_{d-p},..,\chi_{d})$ the $\mathbb{T}$ weights on $\mathbb{C}^{d-p}\times \mathbb{C^*}^{p}.$ Let us denote the action of $G_{\sigma_{1}}$ on $Z_1$ by $(\eta_1,...,\eta_d)\in X(G_{\sigma_1})\cong DG_{\sigma_1}$ and similarly the action of $G_{\sigma_2}$ on $Z_2$ by $(\eta_1^{'},..,\eta_d^{'})\in X(G_{\sigma_{2}})\cong DG_{\sigma_2}$. Let us denote the action of $K$ on $\mathbb{C}^{d-p}\times \mathbb{C^*}^{p}$ by $(\mu_{1},.,\mu_{d-p},..\mu_{d})\in X(K).$ Note that $DG(\beta_{\sigma_i}),i=1,2$ are finite abelian groups. Also note that the morphisms $\pi_2$ and $\pi_1$ are etale and hence the algebra $\mathbb{C}[\tau_1,..\tau_{d-p},\lambda_{1}^{\pm 1},.,\lambda_{p}^{\pm 1}]$ is a free finitely generated module respectively over $$\mathbb{C}[x_{\sigma_{2}(1)},. ,x_{\sigma_{2}(d-p)},x_{\sigma_{2}(d-p+1)}^{\pm 1}.,x_{\sigma_{2}(d)}^{\pm 1}] \ \& \ \mathbb{C}[x_{\sigma_{1}(1)},.,x_{\sigma_{1}(d-p)},x_{\sigma_{1}(d-p+1)}^{\pm 1},.,x_{\sigma_{1}(d)}^{\pm 1}],$$ with bases given by $B_2$ and $B_1$ respectively. Note that they consist of monomials in $\lambda_1,.,\lambda_p$ with appropriate indices.\par
The following lemma shows equal summand on the either sides of the Gluing formula, the total size being $|K|$ on both sides of the equality.
\begin{lem} 
The cardinality of $B_1$ is given by $|B_1|=|DG_{\sigma_2}/DH|$ and likewise $|B_2|=|DG_{\sigma_1}/DH|.$
\end{lem}
\begin{proof}
From the description of the fiber product of two open charts $\mathscr{U}_i,i=1,2$ and the maps $\pi_i,i=1,2$ we first note the following equality of $\mathbb{T}$-weights given by,
$$[\chi][C]:=[\chi_1\ ..\ \chi_{d-p}\ \ \chi_{d-p+1}\\ .. \chi_{d}][C]=[M_1]:=[m^1_{1}..m^1_{d-p}\ m^1_{d-p+1}.. m^1_{d} ]$$ and for the other chart we have,
$$[\chi_1\ ..\ \chi_{d-p}\ \ \chi_{d-p+1}\\ .. \chi_{d}][A]=[M_2]:=[m^2_{1} ..m^2_{d-p}\ m^2_{d-p+1}.. m^2_{d} ].$$\newline
The matrix given by $$[\chi_1\ ..\ \chi_{d-p}\ \ \chi_{d-p+1}\\ .. \chi_{d}]$$ represents the $d\times d$ non-singular matrix whose columns are the characters attached to the torus action on the co-ordinates $(\tau_1,..,\tau_{d-p},\lambda_1,..,\lambda_p).$\newline The matrices $$[m^1_{1}..m^1_{d-p}\ m^1_{d-p+1}.. m^1_{d} ],[m^2_{1} ..m^2_{d-p}\ m^2_{d-p+1}.. m^2_{d} ]$$ are non-singular $d\times d$ matrices whose columns are characters attached to the torus action on the co-ordinates 
$(x_{\sigma_{1}(1)},..x_{\sigma_{1}(d)}),(x_{\sigma_{2}(1)},..x_{\sigma_{2}(d)})$ corresponding to two open charts $\mathscr{U}_i,i=1,2$ respectively.\par
We describe the matrix $C$ which has the form of the block matrix given by,
$$C=\left(\begin{array}{cc} Id_{d-p \times d-p} & 0\\ C'_{p\times d-p} & C_{p\times p} \end{array}\right).$$
The matrix $C'$ has columns that are given by the characters that define the maps $$\alpha_{\sigma_{1}(i)}:(\mathbb{C}^{*})^{p}\to \mathbb{C}^{*}, \forall i=1,..,d-p.$$ Similarly the columns of the non-singular matrix $C$ are given by the characters associated to the following map,
$$\alpha_{\sigma_{1}(i)}:(\mathbb{C}^{*})^{p}\to \mathbb{C}^{*}, \forall i=d-p+1,..,d.$$\newline
The non-singular $d\times d$ matrix $$A=\left(\begin{array}{cc} Id_{d-p \times d-p} & 0\\ 0 & A_{p\times p} \end{array}\right)$$ has the block form evident from the description of $\pi_2.$ The columns of $A_{p\times p}$ are given by the characters associated with the following maps
$$\alpha_{\sigma_{2}(i)}:(\mathbb{C}^{*})^{p}\to \mathbb{C}^{*}, \forall i=d-p+1,..,d.$$
Note that the span of the matrix $C'$ is generated by the non-singular matrix $C_{p\times p}$ and its cokernel. This tells us that the elements of the basis $B_1$ are bijective to the cokernel of the matrix $C_{p\times p}.$ Similarly, the elements of the basis $B_2$ are bijective with the cokernel of the matrix $A_{p\times p}.$\par 
Now from the previous lemma \ref{Cokernel} we consider the morphism
$$(\mathbb{Z}^{d+p})^{*}\xrightarrow{{p_{r}}^{*}}(\mathbb{Z}^{d+r+p})^{*}\xrightarrow{q} DH\bigoplus (\mathbb{Z}^{p})^{*}\xrightarrow{p} (\mathbb{Z}^{p})^{*}$$ where $p$ is the projection onto the second factor $\mathbb{Z}^{p}$ and $p_{r}:\mathbb{Z}^{d+r+p}\to \mathbb{Z}^{d+p}$ projects onto the third and first factor. We denote the composition as $$\tilde{p}:(\mathbb{Z}^{d+p})^{*}\to (\mathbb{Z}^{p})^{*}.$$
Note that $\tilde{p}$ is given by the matrix $[C'_{p\times d-p}|C_{p\times p}|A_{p\times p}].$\par
We have the following exact sequence,
$$0\to(\mathbb{Z}^{d+r})^{*}\to (\mathbb{Z}^{d+r+p})^{*}\xrightarrow{p\circ q} (\mathbb{Z}^{p})^{*}\to 0$$ and under the commutativity of the projection maps we have the following diagram of exact sequences given by,
\begin{center}
~~~~~~\xymatrix{
  0 \ar[r] & (\mathbb{Z}^{d+r})^{*} \ar[d]_-{\id} \ar[r]^-{} & (\mathbb{Z}^{d+r+p})^{*} \ar[d]_-{i^*} \ar[r]^-{p\circ q} & (\mathbb{Z}^{p})^{*} \ar[d]^-{i^*} \ar[r] & 0 \\
  0 \ar[r] & (\mathbb{Z}^{d+r})^{*} \ar[r]_-{} & (\mathbb{Z}^{d+r})^{*} \ar[r]_-{p\circ q} & DG_{\sigma_1}/DH \ar[r] & 0
}
\end{center}
where $p:DG_{\sigma_{1}}\to DG_{\sigma_1}/DH$ denotes the projection map.\par 
Applying the Snake lemma to the above diagram, one achieves the following exact sequence,
$$(\mathbb{Z}^{p})^{*}\xrightarrow{[A_{p\times p}]} ker(i^{*})\to 0.$$ This implies $$|B_2|=|DG_{\sigma_1}/DH|$$ as the following holds,
$$0\to (\mathbb{Z}^{p})^{*}\xrightarrow{[A_{p\times p}]}(\mathbb{Z}^{p})^{*}\xrightarrow{i^{*}} DG_{\sigma_1}/DH\to 0.$$
For the other top cone, we have a similar diagram using the Snake Lemma and the exact sequence we arrive at is given by
$$(\mathbb{Z}^{p})^{*}\xrightarrow{[C_{p\times p}]} ker (i^{*})\to 0.$$
This again implies $$0\to (\mathbb{Z}^{p})^{*}\xrightarrow{[C_{p\times p}]}(\mathbb{Z}^{p})^{*}\xrightarrow{i^{*}} DG_{\sigma_2}/DH\to 0$$ from which we deduce $$|B_1|=|DG_{\sigma_2}/DH|.$$
\end{proof}
For the next proposition, we denote the torus weights for the co-ordinates of the substack $\mathscr{U}_{ij}:=\mathscr{U}_i\times_{\sX}\mathscr{U}_j:=[\mathbb{C}^{d-p}\times (\mathbb{C}^{*})^p/K_{i,j}]$ by $\chi_1,..,\chi_d.$ For the definition of $Q\in Box(\chi_1,..,\chi_d)$ see (Definition \ref{Box}). Denote by $$ S_{\sigma_{i,ij,Q}}:=\{Q'\in B_{\sigma_{i}}(\mathbb{T})\bigg|Q'=(Q-\sum_{k=1}^{p}a_k\chi_{d-p+k}),\prod_{k=1}^{p}\lambda_k^{a_k} \ \in B_{i}\}.$$ Here we denote by $(R)\ , R\in \mathbb{Z}^d$ the fractional part of $R$ (in each coordinate it is a rational number $0\leq r<1$) in the box decomposition of $R$ with respect to the matrix $[m^i_1,..,m^i_d]$ whose columns are given by the characters of the torus action on $\mathscr{U}_{\sigma_i}.$ We denote by $B_i$ the basis (of the free module) of the co-ordinate ring of $\mathscr{U}_{ij}$ over the co-ordinate ring of $\mathscr{U}_i.$ Tensoring a representation by a character represents taking the tensor product of a group representation by a 1 dimensional representation of the group. We can consider a $G_{\sigma_i}$ representation as a $G_{\sigma_i}\times G_{\sigma_j}$ representation via the canonical projection maps. Then we consider a $G_{\sigma_i}\times G_{\sigma_j}$ representation as a $K_{i,j}$ representation through the injection $K_{i,j}\to G_{\sigma_i}\times G_{\sigma_j}$ similar to (\ref{eq:reduction}).
\begin{thm}\label{Gluing Formula}
Let $\sX$ be a smooth toric DM stack given by $(N,\Sigma,\beta).$ 
The category of torsion free toric sheaves is equivalent to the category of $l$-tuples $\{\hat{F}_i\}_{i=1,.,l}$ of finite stacky $S$- families on $\mathscr{U}_{i=1,.,l}$ satisfying the following equality of $K_{i,j}$ representations given at each point $Q+\sum_{i=1}^{d-p}l_i\chi_i$ where $Q\in Box(\chi_1,..,\chi_d),\ i\neq j, \ (l_1,..,l_{d-p})\in \mathbb{Z}^{d-p}$ by,

$$\bigoplus_{\begin{array}{cc}
     &  l\in DG_{\sigma_i}\\
     & Q'\in S_{\sigma_{i,ij,Q}}
\end{array}} ({}_{Q'}F_{i}(l_1,.,l_{d-p},\infty,.,\infty)_{l} \otimes\sum_{i=d-p+1}^{d}(l_{i}\eta_{i},0))\ \otimes\sum_{i=d-p+1}^{d}a_{i-d+p}\mu_{i}$$
$$=\bigoplus_{{\begin{array}{cc}
     &  l\in DG_{\sigma_j}\\
     & Q'\in S_{\sigma_{j,ij,Q}}
\end{array}}}({}_{Q'}F_{j}(l_1,.,l_{d-p},\infty,.,\infty)_{l}\otimes\sum_{i=d-p+1}^{d}\ (0,l'_i\eta'_i))\ \otimes \sum_{i=d-p+1}^{d}a'_{i-d+p}\mu_{i}\ $$
Similar gluing conditions hold for the morphisms.
\end{thm}

\begin{proof}
We prove the theorem for i=1 and j=2 and can be easily generalized to other top cones. In the rest of the proof $K_{i,j}$ is basically denoted as $K_{1,2}=K_{2,1}=K.$
We abbreviate $S_{\sigma_{j,ij,Q}}$ with $S_{\sigma_1,Q}$ and similarly for $\sigma_2.$ We use the same set of torus weights $(\chi_1,..,\chi_d)$ for $\mathscr{U}_{12}.$\par 
Consider $\mathscr{F}$ restricted to $Z_1$ denoted by $\mathscr{F}_1.$ Let us denote by ${}_{(q_1,..,q_d)}F_1(l_1,,l_d)$ the weight sum decomposition of the sheaf indexed by $(q_1,,q_d)\in B_{\sigma_{1}}(\mathbb{T})$ given by the basis of $\mathbb{T}-$ weights on $Z_1.$ On restricting $\mathscr{F}_1$ to $\mathbb{C}^{d-p}\times \mathbb{C^*}^p$ we obtain that for fixed $(l_1,.,l_{d-p})$ the vector spaces ${}_{(q_1,..q_d)}F_1(l_1,,.,l_d)$ stabilize for $l_{d-p+1}>>0,..,l_{d}>>0.$
Upto isomorphism of the stacky finite S-family $\hat{F_1}$ we can take the isomorphisms between the limiting spaces as identity and then denote the limit by: $${}_{(q_1,..,q_d)}F_1(l_1,..l_{d-p},\infty,..,\infty).$$ Consider the $S-$ family corresponding to the sheaf $\mathscr{F}_1$ restricted to $\mathbb{C}^{d-p}\times \mathbb{C^*}^p$ say $\hat{G_1}$ and from the description of graded tensor product one obtains:
$${}_{(q_1,.,q_d)}G_{1}(l_1,.,l_d)={}_{(q_1,.q_d)}F_{1}(l_1,.,l_{d-p},\infty,.,\infty).$$ For sufficiently large $l_{d-p+1},.,l_d$ we have $${}_{(q_1,.,q_d)}F_1(l_1,.,l_d)={}_{(q_1,.,q_d)}F_1(l_1,.,l_{d-p},\infty,.,\infty)$$ and hence the fine grading is obtained as $l\in DG_{\sigma_1}$:
$${}_{(q_1,.,q_d)}G_{1}(l_1,.,l_d)_l={}_{(q_1,.,q_d)}F_{1}(l_1,.,l_d)_l.$$ The fine grading for a fixed $(l_{d-p+1},.,l_{d})$ determines the fine grading for any collection of $(l_{d-p+1},.,l_{d})$ given by:
$${}_{(q_1,.,q_d)}G_1(l_1,.,l_d)_{l}={}_{(q_1,.,q_d)}G_1(l_1,.,l_{d-p},0,.,0)_{l-\sum_{i=d-p+1}^{d}\eta_il_i}\otimes \eta_{d-p+1}^{l_{d-p+1}}..\otimes \eta_{d}^{l_{d}}.$$
The above notation means that the character $$\eta_{d-p+1}^{l_{d-p+1}}\otimes.. \eta_{d}^{l_{d}}$$ is a function from $G_{\sigma_{1}}$ to $\mathbb{C}^{*}$ and the representation refers to the action of an element of $G_{\sigma_{1}}$ on the graded vector space by this character.
Let us define the fine grading as $${}_{(q_1,.,q_d)}F_1(l_1,.,l_{d-p},\infty,.,\infty)_{l}:={}_{(q_1,.,q_d)}G_{1}(l_1,.,l_{d-p},0.,0)_{l}.$$\par
Consider $Q=(q_1,.,q_d)\in B_{\sigma_1}(\mathbb{T}),$ 
and let us describe $\hat{F}_{1,12}$ at the point $\sum_{i=1}^{d}(q_i+l_i)m^1_i.$\newline
Using the description of the graded tensor products, an element of $\hat{F}_{1,12}$ at the said point can be uniquely described as $$\bigoplus_{Q'\in S_{\sigma_1,Q}}(v_{Q'}\otimes \prod_{i=1}^{p}\lambda_i^{a_i})\ $$ where $$v_{Q'}\in {}_{Q'}G_{1}(l_1,.,l_d)$$ and %
the homogeneous elements of the summands are elements of the module $$G_1\otimes_{\mathbb{C}[x_{\sigma_{1}(1)},.,x_{\sigma_{1}(d)}]} \mathbb{C}\bigg[\tau_1,.,\tau_{d-p},\lambda_{1}^{\pm 1},.,\lambda_{p}^{\pm 1}\bigg] .$$\par
The above summation can be explained by noting that the indexing set of the summation is given by
$$S_{\sigma_1,Q}:=\big \{Q'\in B_{\sigma_{1}}(\mathbb{T})\bigg|Q'=(Q-\sum_{i=1}^{p}a_{i}\chi_{d-p+i})  ,\ \prod_{i=1}^{p}\lambda_{i}^{a_i}\in B_1  \big \}.$$
Given an element of $B_1$, given by $(a_1,..,a_p)$ the corresponding $Q'$ is obtained by considering the box decomposition of the integer $Q-\sum_{i=1}^{p}a_{i}\chi_{d-p+i}$ with respect to the non-singular matrix given by $(m^1_1,..m^1_d).$\par For any two such choice of basis elements $\{a_i\},\{a'_i\},i=1,..,p$ corresponding to the same $Q'$ one can show that $$\begin{pmatrix}
  a_{1}-a'_{1} \\
  \vdots \\
  a_{p}-a'_p \\
\end{pmatrix}=[C_{p\times p}]\begin{pmatrix}
 l_1\\
 \vdots\\
 l_p\\
\end{pmatrix}$$ 
where $l_i \in \mathbb{Z}.$ Hence $\{a_i\},\{a'_i\}$ correspond to the same elements in the basis and hence the indexing set of the summand is in bijection with $|S_{\sigma_1,Q}|=|B_1|=|DG_{\sigma_2}/DH|.$

Let us compute the fine grading of the pullback module as a $K$ representation. Assuming $v_{Q'}$ is homogeneous of weight $l\in DG_{\sigma_1},$ one obtains the $X(K)$ weight of the element $v_{Q'}$ given by $\overline{l},$ the image of $l\in X(G_{\sigma_1})$ under the natural surjective map from $X(G_{\sigma_1})$ to $X(K)$ via the embedding into $X(G_{\sigma_1}\times G_{\sigma_2})$ (in [\ref{cokernel}]). Thus we obtain the $X(K)$ grading of $(v_{Q'}\otimes \prod_{i=1}^{p}\lambda_{i}^{a_i})$ given by $$(\overline{l}+\sum_{i=1}^{p}a_{i}\mu_{d-p+i}).$$ Hence, one obtains the value of $\hat{F}_{1,12}$ at the point $$\sum_{i=1}^{d}(q_i+l_i)m^1_i\ \in X(\mathbb{T})$$ given as a $X(K)$ representation by:
$$\bigoplus_{l\in DG_{\sigma_1}}\bigoplus_{\quad Q'\in S_{\sigma_1,Q}}{}_{Q'}G_{1}(l_1,.,l_d)_{l}\otimes {\sum_{i=1}^{p}a_i\mu_{d-p+i}}.$$ The above can be rewritten as:
$$\bigoplus_{\begin{array}{cc}
     &  l\in DG_{\sigma_1}\\
     & Q'\in S_{\sigma_1,Q}
\end{array}} \big ({}_{Q'}F_{1}(l_1,.,l_{d-p},\infty,.,\infty)_{l} \otimes_{i=1}^{p}\eta_{d-p+i}^{l_{d-p+i}}\ \bigl)\otimes \sum_{i=1}^{p}a_i\mu_{d-p+i}.$$\par
Consider $\mathscr{F}$ and restrict it to $Z_2$ and let us compute the value of $\hat{F}_{2,12}$ at the point $\sum_{i=1}^{d}(p_i+l'_i)m^2_{i}.$\newline
Consider $P=(p_1,.,p_d)\in B_{\sigma_2}(\mathbb{T})$ and consider the set $$\bigl\{Q'\in B_{\sigma_2}(\mathbb{T})|Q'=(P-\sum_{i=1}^{p}a'_i\chi_{d-p+i}),\prod_{i=1}^{p}\lambda_{i}^{a_{i}'}\in B_{2}\bigl\}.$$ Thus, applying similar techniques as before, the value at the said point is given by:
$$\bigoplus_{{\begin{array}{cc}
     &  l\in DG_{\sigma_2}\\
     & Q'\in S_{\sigma_2,Q}
\end{array}}}\big ({}_{Q'}F_{2}(l_1^{'},.,l_{d-p}^{'},\infty,.,\infty)_{l}\otimes _{i=1}^{p}(\eta'_{d-p+i})^{l_{d-p+i}'}\big )\otimes \sum_{i=1}^{p}a_i^{'}\mu_{d-p+i}.$$
\par 
We denote the equivalence of categories up to isomorphism between the finely graded $S-$ families and the isomorphism is taken as identity and the Gluing Formula is given at any point as the equality of $K$ representations determined as above. \par
From the description of $\hat{F}_{2,12}$ we see that it is enough to compute the Gluing formula at the point $$Q:=[\chi_1,..,\chi_{d-p}\  \chi_{d-p+1},..,\chi_{d}]\begin{pmatrix}
  q_1+l_1 \\
  \vdots \\
  q_{d-p}+l_{d-p} \\
  q_{d-p+1} \\
  \vdots \\
  q_{d} \\
\end{pmatrix}=\sum_{i=1}^{d-p}(q_i+l_i)\chi_i+\sum_{i=d-p+1}^{d}q_i\chi_i.$$ Here $(q_1,..,q_d)\in \text{Box}(\chi_1,..,\chi_d),\ l_i\in \mathbb{Z}.$\newline
Thus for $\hat{F}_{1,12}$ enough to evaluate at the point given by,
$$[C]^{-1}\begin{pmatrix}
  q_1+l_1 \\
  \vdots \\
  q_{d-p}+l_{d-p} \\
  q_{d-p+1} \\
  \vdots \\
  q_{d} \\
\end{pmatrix}=\left(\begin{array}{cc} Id_{d-p \times d-p} & 0\\ -C^{-1}_{p\times p}C'_{p\times d-p} & C^{-1}_{p\times p} \end{array}\right)\begin{pmatrix}
  q_1+l_1 \\
  \vdots \\
  q_{d-p}+l_{d-p} \\
  q_{d-p+1} \\
  \vdots \\
  q_{d} \\
\end{pmatrix}$$ with respect to the basis $(m^1_1,..,m^1_d).$ \newline 
Under the box decomposition, the co-ordinates, say are given by $$(q_1,..,q_{d-p},t_{1},..,t_p)\in B_{\sigma_1}(\mathbb{T})$$ and $$(l_1,..,l_{d-p},l_{d-p+1},..,l_d)\in \mathbb{Z}^d.$$
Similarly, the same computation for the other top cone lets us compute $\hat{F}_{2,12}$, given by the co-ordinates
$$(q_1,..,q_{d-p},t'_1,..,t'_p)\in B_{\sigma_2}(\mathbb{T}),$$ and $$(l_1,..,l_{d-p},l'_{d-p+1},..,l'_{d})\in \mathbb{Z}^d$$ with respect to the basis $(m^2_1,..,m^2_d).$\par 
Then we have the following 

\begin{equation}\label{eqn_Gluing}
\begin{aligned}
&\bigoplus_{\begin{array}{cc}
&  l\in DG_{\sigma_1}\\
& Q'\in S_{\sigma_1,Q}
\end{array}} ({}_{Q'}F_{1}(l_1,.,l_{d-p},\infty,.,\infty)_{l} \otimes\ \sum_{i=d-p+1}^{d}\ (l_i\eta_i,0)) \otimes \sum_{i=d-p+1}^{d}a_{i-d+p}\mu_{i} \\
=&\bigoplus_{{\begin{array}{cc}
     &  l\in DG_{\sigma_2}\\
     & Q'\in S_{\sigma_2,Q}
\end{array}}}({}_{Q'}F_{2}(l_1,.,l_{d-p},\infty,.,\infty)_{l}\otimes\  \sum_{i=d-p+1}^{d} \ (0,l'_i\eta'_i))\otimes\ \sum_{i=d-p+1}^{d}a'_{i-d+p}\mu_{i}.\\ 
\end{aligned}
\end{equation}
\end{proof}

\begin{rmk} No gluing formula exists for torsion sheaves.    
\end{rmk}
Next we compute the $\mathbb{T}$ equivariant line bundles on $\sX.$
\begin{cor}\label{LB}
$Pic_{\mathbb{T}}(\sX)\ \cong\ \mathbb{Z}^{n}.$
\end{cor}
\begin{proof}
An equivariant line bundle on $\mathscr{F}$ on $\sX$ is given by the choice of a box element and the position of the homogeneous generator on each open chart $\sigma$ for each top cone with a $X(\mathbb{T})$ grading and a $DG_{\sigma}$ fine-grading. This data has to satisfy the gluing condition.\par
In the chart corresponding to $\sigma_1,\sigma_2$ let us fix the position of the homogeneous generator $(A^1_1,..,A^1_d),(A^2_1,..,A^2_d)$ and $Q_1:=(q^1_1,..,q^1_d),\ Q_2:=(q^2_1,..,q^2_d)$ be the box element and denote $B_1:=Q_1+\sum_{k=1}^{d}A^1_km^1_k,\ B_2:=Q_2+\sum_{k=1}^{d}A^2_km^2_k$ respectively. Note that in terms of the basis $(m^2_1,..,m^2_d)$ we have the following description of $B_1=Q'_1+\sum_{k=1}^{d}A'_km^2_k,$ where 
$$Q'_1=(q^1_1,..,q^1_{d-p},t_1,..,t_p)$$ and $$A'_k=A^1_k\ \forall k=1,..,d-p.$$
By evaluating $\hat{F}_{1,12}(B_1)$ one observes using the Gluing formula that $Q_2\in S_{\sigma_{2,B_1}}$ and $${}_{Q_2}F_{2}(A^1_1,..,A^1_{d-p},A'_{d-p+1},..,A'_{d})_{l-\sum_{i=d-p+1}^{d}A'_{i}\eta'_{i}}$$ is non-zero for some $l\in DG_{\sigma_2}$. Applying the same idea for $B_2$ we conclude that $A^1_k=A^2_k\  \forall k=1,..,d-p.$\newline
Next observe that for $Q_2\in S_{\sigma_2,B_1}$, for some $(a_1,..,a_p)$ corresponding to a basis element of $B_2$ we have the following equality,
$$Q_2=(Q'_1-\sum_{i=1}^{p}a_i\chi_{d-p+i})$$ where $(r)$ represents the fractional part of $r\in \mathbb{Q}$ on each co-ordinate with respect to the matrix $[M_i]$.\newline
This implies
$$(Q_2-Q'_1)=[\chi][A]\begin{pmatrix}
    q^2_1-q^1_1\\
    \vdots\\
    q^2_{d-p}-q^1_{d-p}\\
    q^2_{d-p+1}-t_1\\
    \vdots\\
    q^2_{d}-t_p
\end{pmatrix}=[\chi][A]\bigg(\begin{pmatrix}
    x_1\\
    x_2\\
    \vdots\\
    x_d\\
\end{pmatrix}+[A]^{-1}\begin{pmatrix}
    0\\
    \vdots\\
    a_1\\
    \vdots\\
    a_p
\end{pmatrix}\bigg),$$ for some $(x_1,..,x_d)\in \mathbb{Z}^d.$
Therefore, injectivity forces,
$$q^1_k=q^2_k\ \forall k=1,..,d-p.$$

Hence, for the top cones $\sigma_i,\sigma_j$ the gluing formula shows that $$q^i_k+A^i_k=q^j_k+A^j_k,\ \ \forall k=1,\ldots,d-p,\ \sigma_{i}(k)=\sigma_{j}(k)$$ where $(q^i_1,..,q^i_d),(A^i_1,..,A^i_d)$ are the box and the homogeneous generator, respectively, corresponding to each chart $\sigma_i.$ \par
Let the top cones be indexed in the following order given by $\sigma_i,i=1,..,l.$ For a given top cone, let the rays be indexed by $\sigma_i(j),j=1,..,d.$ We start assigning the integers to $\sigma_1(j),j=1,..,d$ and assign a new integer for $\sigma_i(j)$ if $\sigma_i(j)\neq \sigma_{i'}(j')$ for $i>i'$ and any $j'.$\par 
Let, $$[M_i] \begin{pmatrix}
   q^i_1+A^i_1 \\
   \vdots \\
   q^i_j+A^i_j\\
   \vdots \\
   q^i_d+A^i_d\\
   
\end{pmatrix}=\begin{pmatrix}
    B^i_1\\
    \vdots\\
    B^i_j\\
    \vdots\\
    B^i_d
\end{pmatrix}.$$
Thus the following choices of integers for each ray fix the line bundle given by $$B:=(B^1_1,..,B^1_d,..,B^i_j,..,B^i_d)\in \mathbb{Z}^n$$ where $$1\leq i\leq l,\ 1\leq j \leq d$$ and $$\rho_t=\sigma_{i'}(j')\ \forall t=1,..,n,\ i'\leq i,\ j'\leq d.$$\par Next we show the fine-grading on each open chart is fixed by the given collection of integers.
We prove our claim for $i=1,2$ and the the same procedure generalizes for the other top cones. Let us denote the fine-grading of the homogeneous generators on each chart by $m_1,m_2$ for $\mathscr{U}_i,i=1,2$ respectively. Fix $\sigma_1$ and we have,
$$_{Q_1}F_{1}(A^1_1,.,A^1_d)_{m_1}=\ _{Q_1}F_{1}(A^1_1,.,A^1_{d-p},\infty,.,\infty)_{m_1-\sum_{i=d-p+1}^{d}A^1_i\eta_i}.$$
The fine-grading $m_1\in DG_{\sigma_1},\ m_2\in DG_{\sigma_2}$ are the solutions to the following equation in $X(K).$ We follow (\ref{eqn_Gluing}) to deduce the following,
$$\overline{\bigg(m_1-\sum_{k=d-p+1}^{d}A^1_k\eta_k\ ,\ -m_2+\sum_{k=d-p+1}^{d}A_k^{'}\eta'_k \bigg)}\ =\sum_{k=d-p+1}^{d}a'_{d-p+k}\mu_{k}.$$
From the fact that $X(K)\cong (DG_{\sigma_1}\bigoplus DG_{\sigma_2})/DH$ we see that $\mu_k=\overline{(x_k,y_k)}$ for some $(x_k,y_k)\in DG_{\sigma_1}\bigoplus DG_{\sigma_2}.$ Note that $A^1_k,A'_k,a'_{d-p+k}$ are uniquely determined by the n-tuple of integers $(B^1_1,..,B^1_d,..,B^i_j,..,B^i_d)$ and the matrices $[M_1],[M_2].$\par
If $(m'_1,m'_2)$ is another pair of solution for the above equation in $DG_{\sigma_1}\times DG_{\sigma_2}$ then one observes $$(m_1-m'_1,m_2-m'_2)=(\phi_1(h),\phi_2(h))$$ for some $h\in H.$\par
Therefore,
$$m_1=\sum_{k=d-p+1}^{d}A^1_k\eta_k+ a'_kx_k\ $$ and $$m_2=\sum_{k=d-p+i}^{d}a'_ky_k-A'_k\eta'_k$$ are solutions modulo $H.$ \par
Next let us describe $$DG_{\sigma_1}\cong \prod_{j=1}^{d+r}\mathbb{Z}_{a_j},\ DG_{\sigma_1}/H\cong \prod_{j=1}^{d+r}\mathbb{Z}_{\lambda_j}$$ and 
$$\ DG_{\sigma_2}\cong \prod_{j=1}^{d+r}\mathbb{Z}_{a'_j},\ DG_{\sigma_2}/H\cong \prod_{j=1}^{d+r}\mathbb{Z}_{\lambda'_j}$$ in their invariant factor form. Let us also describe $$m_1=\bigg(m_{1,1}(mod\ \mathbb{Z}_{a_1}),..,m_{1,j}(mod\ \mathbb{Z}_{a_j}),..,m_{1,d+r}(mod\ \mathbb{Z}_{a_{d+r}})\bigg)$$ and $$m_2=\bigg(m_{2,1}(mod\ \mathbb{Z}_{a'_1}),..,m_{2,j}(mod\ \mathbb{Z}_{a'_j}),..,m_{2,d+r}(mod\ \mathbb{Z}_{a'_{d+r}})\bigg)$$ in their invariant factor form. Thus, the following description uniquely determines the fine-grading $m_1,m_2$ in terms of $B$ which are given by,
$$m_1=\bigg(m_{1,1}(mod\ \mathbb{Z}_{\lambda_1}),..,m_{1,j}(mod\ \mathbb{Z}_{\lambda_j}),..,m_{1,d+r}(mod\ \mathbb{Z}_{\lambda_{d+r}})\bigg)$$ and $$m_2=\bigg(m_{2,1}(mod\ \mathbb{Z}_{\lambda'_1}),..,m_{2,j}(mod\ \mathbb{Z}_{\lambda'_j}),..,m_{2,d+r}(mod\ \mathbb{Z}_{\lambda'_{d+r}})\bigg).$$
Thus $B\in \mathbb{Z}^n$ determines a line bundle by $L_{B}$ and hence from the above description we obtain an isomorphism 
\begin{align*}
  f\colon \mathbb{Z}^n & \longrightarrow Pic_{\mathbb{T}}(\sX) \\[-1ex]
  B & \longmapsto L_{B}
\end{align*}
respecting the group law $(B+B')\to L_{B+B'}\cong L_{B}\otimes L_{B'}$ and $(0,..,0)\to \sO_{\sX}, $ (trivial linearization). \end{proof}
\par
In this paper we work with slope-stable torsion free sheaves $\mathscr{F}$ and we want to classify the indecomposable sheaves, as stable implies indecomposable. More precisely, $\mathscr{F}$ is slope-stable if and only if the reflexive sheaf $\mathscr{F}^{**}$ is so. In particular $\mathscr{F}$ is indecomposable if and only if its reflexive hull is indecomposable. The following result will give us more information about the characteristic functions of indecomposable sheaves in the next section.
\begin{thm}\label{Reflexive hull}
Let $\mathscr{F}$ be a toric reflexive sheaf on $\sX$ and on some open chart $\mathscr{U}_{\sigma}$ if  the $S-family,\ \hat{F}^{\sigma}$ corresponding to $\mathscr{F}|_{\mathscr{U}_{\sigma}}$ has at least two box summands, then $\mathscr{F}$ is equivariantly decomposable into toric subsheaves of smaller rank.
\end{thm}
\begin{proof}
We prove the theorem for rank 2. The case of higher ranks can be treated in a similar way. For each open substack $\mathscr{U}_\sigma$ we see that $\mathscr{F}|_{\mathscr{U}_{\sigma}}$ has two or one box summand. Each summand of the sheaf in an open substack is a reflexive sheaf. Therefore, it can be described in the same way using (Section 3,\cite{GJK}). Note that in case of rank 2, the summands are equivariant reflexive sheaves of rank 1 and hence are given by equivariant line bundles. \par
We first ignore the action of the local group on each open chart and glue the sheaves. 
Assume for $\mathscr{U}_{\sigma_i}$ that $S-family,\ \hat{F}_{\sigma_i}$ has two box elements $Q_i,Q'_i$ and the summands corresponding to the ones are given by $_{Q}\hat{F}_{\sigma_i}$ with homogeneous generators given by $(A^i_1,..,A^i_d)$ and $_{Q'}\hat{F'}_{\sigma_i}$ with homogeneous generators given by $(A'^i_1,..,A'^i_d).$ 
Let us represent $$Q_i=(q^i_1,..,q^i_d),\ Q'_i=(q'^i_1,..,q'^i_d).$$ Therefore we have $$\hat{F}_{\sigma_i}=\ _{Q_i}\hat{F}_{\sigma_i}\bigoplus\ _{Q'_i}\hat{F'}_{\sigma_i}.$$
We have, 
$$_{Q_i}\hat{F}_{\sigma_i}(\infty,..,l_k,..,\infty) = \left\{ 
\begin{array}{ll}
0 & \ l_k\ <  A^i_k \ \\
p \cong \mathbb{C} & \ l_k\geq A^i_k.\
\end{array}
\right.$$
Similarly, we have
$$_{Q'_i}\hat{F'}_{\sigma_i}(\infty,..,l_k,..,\infty) =\left\{
\begin{array}{ll}
0 & \ l_k\ <  A'^i_k \ \\[2ex]
p' \cong \mathbb{C} & \ l_k\geq A'^i_k.\ 
\end{array}
\right.$$ Note that $p\neq p' \in \mathbb{P}^1.$\par
For some $\mathscr{U}_{\sigma_j},i\neq j$ let us assume without loss of generality $\sigma_i(k)=\sigma_j(k), k=1,..,d-p.$ There are only two possible cases with which we deal accordingly. \par If there is one box summand for the chart $\mathscr{U}_{\sigma_j}$ given by $$Q_j=(q^i_1,..,q^i_{d-p},q^j_{d-p+1},..,q^j_{d}),$$ then from (\ref{Gluing Formula}) we see that $$q^i_k=q'^i_k,\  k=1,..,d-p.$$
In this case we can describe $S-family,\  \hat{F}_{\sigma_j}$ for $k=1,..,d-p$ as follows,
$$_{Q_j}\hat{F}_{\sigma_j}(\infty,..,l_k,..,\infty)=\left\{ 
\begin{array}{ll}
0 & \ \ l_k\ <  min(A^i_k,A'^i_k) \ \\[2ex]
p \cong \mathbb{C} & \ \ min(A^i_k,A'^i_k)\leq l_k < max(A^i_k,A'^i_k) \ \\[2ex]
\mathbb{C}^2 & \ \   max(A^i_k,A'^i_k)\leq l_k.\ \\
\end{array}
\right.$$ 
For $k=d-p+1,..,d$  we have 
$$_{Q_j}\hat{F}_{\sigma_j}(\infty,..,l_k,..,\infty)=\left\{ 
\begin{array}{ll}
0 & \ \ l_k\ < A^j_k \ \\[2ex]
p' \cong \mathbb{C} & \ \ A^j_k \leq l_k < A^j_k+\lambda^j_k,\  \lambda^j_k \in \mathbb{N}_{>0}\ \\[2ex]
\mathbb{C}^2 & \ \   A^j_k+\lambda^j_k\leq l_k.\ \\
\end{array}
\right.$$ 
If some other chart corresponds to a top cone $\mathscr{U}_{\sigma_{j'}}$ such that $\sigma_j(k)=\sigma_{j'}(k)$ for some $d-p+1\leq k\leq d,$ then two cases are possible. If there is one box summand $Q_j'$ then $q^j_k=q'^j_k$ then the $S-family$ corresponding to the ray $\sigma_{j'}(k)$ are determined by the filtration of $\sigma_j(k).$\par If there are two box summands corresponding to $Q_{j'},Q'_{j'}$ given by $$\hat{F}_{\sigma_{j'}}=\ _{Q_{j'}}\hat{F}_{\sigma_{j'}}\bigoplus\ _{Q'_{j'}}\hat{F'}_{\sigma_{j'}}$$ then $q^{j'}_k=q'^{j'}_k=q^j_k$ and we have the following filtration given by
$$_{Q_{j'}}\hat{F}_{\sigma_{j'}}(\infty,..,l_k,..,\infty)=\left\{ 
\begin{array}{ll}
0 & \ \ l_k\ < A^j_k+\lambda^j_k \ \\[2ex]
p \cong \mathbb{C} & \ \  l_k \geq A^j_k+\lambda^j_k, \ \\[2ex]
\end{array}
\right.$$ 
and for the other summand given by
$$_{Q'_{j'}}\hat{F'}_{\sigma_{j'}}(\infty,..,l_k,..,\infty)=\left\{ 
\begin{array}{ll}
0 & \ \ l_k\ < A^j_k \ \\[2ex]
p' \cong \mathbb{C} & \ \  l_k \geq A^j_k. \ \\[2ex]
\end{array}
\right.$$ 
Thus, if the filtration corresponding to one ray in one chart is determined, then so is the filtration corresponding to the ray in every chart. In this case the reflexive sheaf decomposes as direct sum of two equivariant line bundles given by the integers at which the filtration jumps for each ray. The local group action on the weight spaces of $\hat F_{\sigma_j}$ are uniquely determined by (\ref{Gluing Formula}) as in the proof of (\ref{LB}).\par
If there are two box summands for $\mathscr{U}_{\sigma_j}$ then denote them by $Q_j,Q'_{j}.$ For $1\leq k \leq d-p$ let, $q^i_k\neq q'^i_k.$ In this case from (\ref{Gluing Formula}) we have without loss of generality, $q^i_k=q^j_k$ and $q'^i_k=q'^j_k.$ Then the box summands correspond according to the box elements. When $q^i_k=q^j_k=q'^i_k=q'^j_k$ from (\ref{Gluing Formula}) the filtration for each box summand can be determined from the position of the homogeneous generator on each summand corresponding to $k.$ In both cases the reflexive sheaf decomposes equivariantly as direct sum of line bundles and the local group action on the weight spaces are determined as before. 
\end{proof}
\begin{example}
The above discussion allows us to classify stable reflexive sheaves of rank 2 on any smooth toric DM stack by studying indecomposability. First ignore the local group action and glue the sheaf. Similarly, to each chart, we have one fixed box element given by $Q_i=(q^i_1,..,q^i_d)$. Note that $q^i_k=q^j_k,\ \text{if}\  \sigma_i(k)=\sigma_j(k).$  On that chart the filtration is given as follows,
$$_{Q_i}\hat{F}_{\sigma_i}(\infty,..,l_k,..,\infty)=\left\{ 
\begin{array}{ll}
0 & \ \ l_k\ < A^i_k \ \\[2ex]
p_k \cong \mathbb{C} & \ \ A^i_k \leq l_k < A^i_k+\lambda^i_k,\  \lambda^i_k \in \mathbb{N}_{>0}\ \\[2ex]
\mathbb{C}^2 & \ \   A^i_k+\lambda^i_k\leq l_k.\ \\
\end{array}
\right.$$ 
Using (\ref{Gluing Formula}) we see that each sheaf is determined by the following data (one for each ray) given by $(A_1,..A_n),(\lambda_1,..,\lambda_n)\in (\mathbb{N}_{\geq 0})^n,(p_1,..,p_n)\in (\mathbb{P}^1)^n.$ $\lambda_i$ are allowed to be $0.$ Moreover, the local group action on each weight space is uniquely determined as in the proof of 
(\ref{LB}). In this case, the limiting vector space may not decompose as a direct sum of weight spaces of dimension $1$ with different weights.\par
In the case of smooth toric DM surfaces, reflexive sheaves being locally free, from Klyachko's splitting criterion any rank r toric vector bundle (given by one box element on each chart) decomposes if the filtration corresponding to all rays shares a common basis element for $\mathbb{C}^r.$\par 
Similarly for equivariant toric reflexive sheaves, decomposition into smaller rank equivariant reflexive sheaves occur when all filtration corresponding to each ray shares common basis elements. 

\end{example}
\section{Fixed point Locus}
\subsection{Characteristic functions}\label{Char}
In order to understand the toric structure on torsion free sheaves on $\sX$ we introduce certain functions which describe the dimension of the  weight spaces on each chart corresponding to a character of the torus. In essence, the shape of this function captures the module structure of the torsion free sheaf and allows us to construct moduli spaces using GIT, parametrizing toric torsion free sheaves, such that this function stays invariant. \par  Let $\mathscr{F}$ be a torsion free sheaf on a toric DM stack $\sX$ and we define the corresponding characteristic function as 
$${\vec\chi_{\mathscr{F}}}= \coprod \limits_{b_1\in B_{\sigma_{1}}(\mathbb{T})}..\coprod\limits_{b_l\in B_{\sigma_l}(\mathbb{T})}\ _{(b_1,..,b_l)}\vec \chi_{\mathscr{F}}$$ where,
$$_{(b_1,..,b_l)}\vec\chi_{\mathscr{F}}:(\mathbb{Z}^d)^{l}\to \mathbb{Z}^l$$ given by,
$$\big(\ _{(b_1,..,b_l)}\chi^{\sigma_1}_{\mathscr{F}}(m_1),..,\ _{(b_1,..,b_l)}\chi^{\sigma_l}_{\mathscr{F}}(m_l)\ \big)$$
$$=\big(\mathsf{dim}_{\mathbb{C}}(_{b_1}F^{\sigma_{1}}(m_1)),..,\mathsf{dim}_{\mathbb{C}}(_{b_l}F^{\sigma_l}(m_l))\big)$$ given by restricting to all $l$- tuples $(b_1,..,b_l)$ satisfying the Gluing formula (\ref{Gluing Formula}) and denote the subset of $l$- tuples fixed by the Gluing formula by $B.$\par
\begin{rmk}\label{Stable remark}
Next we focus on the case of stable sheaves. Stability implies indecomposable sheaves. From (\ref{Reflexive hull}) we see that indecomposable reflexive sheaves have only one $b_i$ non-zero for the box summand corresponding to each top cone $\sigma_i.$ This forces the characteristic functions of stable torison free sheaves to be of the form,
$${\vec\chi_{\mathscr{F}}}=\ _{(b_1,..,b_l)}\vec \chi_{\mathscr{F}}.$$
\end{rmk}
As in (\cite{Koo2}) we can show that for a torsion free sheaf the characteristic function is a finer invariant and fixing one fixes the modified Hilbert polynomial with respect to the polarizations.\par
Defining $Gr(m,n)$ as $m$ dimensional subspaces in $\mathbb C^{n},$ we define the following quasi projective variety so that we can realize the torsion free framed $S$- family with characteristic function $\vec\chi$ as closed points of a locally closed subscheme of this ambient variety given by (we restrict to $(b_1,,b_l)$ satisfying the above compatibility),
$$\mathscr{A}=\coprod\limits_{(b_1,..,b_l)\in B}\big (\prod\limits_{i=1}^{l}\prod\limits_{m_i\in \mathbb{Z}^d} Gr(_{(b_1,..,b_l)}\chi^{\sigma_{i}}_{\mathscr{F}}(m_i),r)\ \big).$$ Denote the locally closed subscheme by $\mathscr{N}_{\vec \chi}.$ There is an action of $G:=SL(r,\mathbb C)$ on $\mathscr{A}$ leaving $\mathscr{N}_{\vec\chi}$ invariant. For any choice of $G$- equivariant line bundle $\mathscr{L}\in Pic^{G}(\mathscr{N}_{\vec\chi})$ we get the notion of GIT (semi)stable elements of $\mathscr{N}_{\vec\chi}$ with respect to $\mathscr{L}$, (\cite{Dol}). Let us denote the $G$ invariant subset of (semi)stable points by $(\mathscr{N}^{ss}_{\vec\chi})
\ \mathscr{N}^{s}_{\vec\chi}.$\\\par
We declare a torsion free sheaf $\mathscr{F}$ with characteristic function $\vec\chi$ as GIT (semi)stable if its corresponding framed $S$-family is such (see in \cite{Koo1}, \cite{Koo2}, \cite{Koo3}, \cite{WW}). We can form the corresponding categorical/ geometric quotients for the $G$ action on the semi-stable and stable locus and obtain the moduli spaces as quasi projective scheme, $(M^{ss}_{\vec\chi})\ M^{s}_{\vec\chi}$, the latter being Zariski-open inside the former.\par
Let $(\mathscr{M}^{ss}_{\vec\chi})\mathscr{M}^{s}_{\vec\chi}$ be the moduli functor of modified Gieseker (semi)stable torsion free sheaves with $\vec\chi$ as characteristic function with respect to $\sO_{X}(1)$ and $\Xi$, both equipped with torus equivariant structures. Under the existence of a line bundle matching GIT and modified Gieseker semi-stability we have similar results to (\cite{Koo2}, Theorem 3.11, 3.12)
stating that $(\mathscr{M}^{ss}_{\vec\chi})\mathscr{M}^{s}_{\vec\chi}$ is corepresented by $(M^{ss}_{\vec\chi})\ M^{s}_{\vec\chi}.$ Here $M^{s}_{\vec\chi}$ is the coarse moduli space of the functor $\mathscr{M}^s_{\vec\chi}$.\\
\subsection{Matching Stability}\label{GIT and Gieseker}
From this section onward we mostly concentrate on the case of smooth toric DM orbifolds.
In \cite{Koo2} we have that the purity of a toric coherent sheaf can be tested by testing the purity of the equivariant coherent subsheaves. Then  they show that for fixed equivariant polarization data Gieseker semi-stability and stability can be checked by studying the corresponding property of their subsheaves. Throughout the proof, the recipe they make use of is the existence of Krull-Schmidt decomposition property in the category of coherent sheaves on the smooth toric variety, connectedness of $\mathbb{T}^{cl}$ and (\cite{Per}, Theorem 2.30) which is the Complete Reducibility Theorem. All of the above recipes hold in our case as well and we infer that in our case it is enough to test the modified Gieseker (semi)stability with respect to equivariant polarization $(H,\Xi)$ of torsion free sheaves by studying their coherent subsheaves.\par
Recall (\ref{Reflexive hull}), for stable sheaves we are working with. With one box summand per chart, stable equivariant torsion free sheaf admits a finite locally free equivariant resolution on a smooth toric DM stack.
Under this condition we can match the GIT stability and the Gieseker stability as in (\cite{Koo2}, Thm.3.21).
\begin{thm}\label{Gieseker}
Let $\sX$ be a smooth projective orbifold toric DM stack with a polarization $(H,\Xi)$ given by an ample line bundle and an ample equivariant generating sheaf, respectively. Then for any $\vec\chi$ there is an ample equivariant line bundle $\mathscr{L}\in Pic^{G}(\mathscr{N}_{\vec\chi})$ such that any torsion free sheaf on $\sX$ with characteristic function $\vec\chi$ is GIT stable with respect to $\mathscr{L}$ if and only if modified Gieseker stable with respect to the fixed polarizations.
\end{thm}
\begin{proof}
Under the observation that stable torsion free sheaves have one box summand per chart, we directly follow the proof of \cite{Koo2} and stick to the notation as there and state only the necessary changes. First note that tensoring by a suitable line bundle the box elements can be made zero on each co-ordinate without changing Gieseker stability.
First we observe that in the modified Hilbert polynomial (\ref{defn_Gieseker}) we have the chern character of the generating sheaf. In the proof we fix $(0,..,0)$ as the box element and then the proof revolves around showing that as a polynomial in $t,\  \Xi_{\sigma_{i',j',k'},\vec\lambda}(t)dim(_{b_1}F^{\sigma_i'}(\lambda_1,..,\lambda_j',\infty,..,\infty))$ is a polynomial of degree $r-j'$ with positive leading co-efficient.\par
Assume the stack is an orbifold, then we observe that in our proof we need the Riemann-Roch theorem (\ref{eqn_orbifold_RR}) and (\cite{Edidin}) to argue that we only consider the terms supported on the main component $\sX$ and in order to compute the maximal degree of the polynomial in $t,$ we need only the constant term from chern character (Equation \ref{eqn_chern}) of the generating sheaf. For Todd class contributions check (Equation \ref{Todd}).\par
To show the positivity of the leading coefficient we can use the Nakai Mosheizon criterion and the fan associated to the stack being simplicial, the same calculations go through.\par
The box elements being one per chart it is easy to see that the same sequence of positive integers given by $\Xi_{\sigma_{i',j',k'},\vec\lambda}(R),R>>0$ gives us the necessary linearization on the products of Grassmanians. Restricting to the locally closed subscheme and following the same argument, we achieve the result.
\end{proof}
Next, we prove a similar result for matching modified slope stability (Definition \ref{defn_modified_slope_stability}) and GIT stability as in (\cite{Koo2},Thm 3.20). The line bundles $\mathscr{L}^\mu\in \ Pic^{G}(\mathscr{N}_{\vec\chi})$ are different to $\mathscr{L}$ and are easier to construct.\par
Note that modified slope stability on smooth projective DM stacks do not depend on the choice of the generating sheaf $\Xi$ (Definition \ref{defn_generating_sheaf}). In other words if $\mathscr{F}$ is modified semi-stable with respect to $(H,\Xi)$ then so it is semi-stable with respect to another $(H,\Xi').$ Using (\cite{Nir}. Prop. 3.18. Remark 3.19) we know that we can always choose a ample generating sheaf on a smooth toric orbifold of the form $\bigoplus_{i=1}^{n}\mathscr{O}_{\sX}(\mathscr{D}_i),$ where $\mathscr{D}_i$ are the invariant toric divisors corresponding to $n$ rays of the fan (\ref{rmk_stability_quotient}). We have the following result corresponding to the generating sheaf $\Xi:=\ \bigoplus_{i=1}^{n}\mathscr{O}_{\sX}(\mathscr{D}_i)$.
\begin{thm}\label{slope stability}
Let $\sX$ be a smooth projective toric DM stack and let  $H$ be an
ample line bundle and $\Xi$ be a ample equivariant generating sheaf on $\sX$. Then for any $\vec\chi,$ there is an ample equivariant line bundle
$\mathscr{L}^\mu\in\  Pic^G(\mathscr{N}_{\vec\chi})$, such that any torsion free equivariant sheaf $\mathscr{F}$ on $\sX$ with characteristic 
function $\vec\chi$ satisfies
$\mathscr{F}$ is modified slope-stable $\implies \ \mathscr{F}$  is properly GIT stable with respect to $\mathscr{L}$ $\implies\ \mathscr{F} $ is modified slope-semi-stable.    
\end{thm}
\begin{proof}
The proof of the above result follows the same construction as in (\cite{Koo2}, Prop 3.20) as stable sheaves are given by one box summand per chart and on each of the charts the torsion free sheaves are given by multifiltration of $\mathbb{C}^{\bigoplus r}$. We stick to the same notation and address the changes. \newline We obtain filtration of $\mathbb{C}^{\bigoplus r}$ given by $_{q^i_j}\{\beta_{\lambda}\}_{\lambda \in \mathbb{Z}}$ corresponding to $\rho_j\in \sigma_i$ given as
$$_{q^i_j}\beta_\lambda=\left\{ 
\begin{array}{ll}
0 & \ \ \lambda\ < A_j\ \\[2ex]
p_k \cong \mathbb{C} & \ \ A_j \leq \lambda < A_j+\Delta_j(1),\  \Delta_j(1) \in \mathbb{N}_{\geq0}\ \\[2ex]
\vdots\\
\mathbb{C}^r & \ \   A_j+\Delta_j(1)+..+\Delta_j(r-1)\leq \lambda,\ \Delta_j(k) \in \mathbb{N}_{\geq0}.\\\
\end{array}
\right.$$ 
The above filtration is irrespective of the top cone $\rho_j\in \sigma_i$ using (\ref{Gluing Formula}). Note that $b_i=(q^i_1,..q^i_j,..q^i_d)$ is a box element corresponding to the top cone $\sigma_i.$ For any other top cone where $\rho_j$ appears, $q^i_j$ stays unchanged from (\ref{Gluing Formula}). Tensoring by an equivariant line bundle does not change the slope (semi) stability and hence we can assume for the rest of the proof that the box elements corresponding to each chart for the $S-family,\ \hat F$ is given by $(0,..,0).$\par
The rest of the proof follows from the fact that the modified slope is determined by $c_1(\mathscr{F}).H/r$ (use \ref{eqn_orbifold_RR}). 
Then,
$$c_1(\mathscr{F})=\sum_{j=1}^{n}(rA_j+(r-1)\Delta_{j}(1)..+\Delta_j(r-1))(c_1(\mathscr{D}_j)).$$ Similarly $\Xi,\pi^*H$ can be both expressed in terms of $c_1(\mathscr{D}_j),\ 1\leq j \leq n.$
The construction of the equivariant line bundle follows from the same recipe as the rest of the proof is the same.
\end{proof}
\subsection{Fixed Point Locus}
For a fixed Hilbert polynomial $P_{\Xi}$ we denote the moduli space of modified Gieseker semi-stable torsion free sheaves by $M^{ss}_{P_{\Xi}}$ with respect to the polarization $(H,\Xi)$ which, admits a natural induced regular torus action. The same torus action restricts to the stable locus and is given by the same expression as in (\cite{Koo2}, Prop 4.1) on the closed points.\par
Assume that we can always pick a $G$- equivariant line bundle matching GIT and modified Gieseker semi-stability we have a forgetful natural transformation between the functors given by
$$\mathscr{M}^{s}_{\vec\chi} \to \mathscr{M}^{s}_{P_{\Xi}}$$ and similar in the case of semi-stable ones.
Denote by $$\vec\chi\in \chi_{P_{\Xi}}$$ the characteristic functions that give rise to the same modified Hilbert polynomial $P_{\Xi}.$
The above natural transformation gives rise to natural morphisms between the moduli spaces co-representing them, which, on the closed points is just the natural morphism forgetting the equivariant structure given by
$$\coprod\limits_{\chi\in\chi_{P_{\Xi}}}M^{s}_{\vec\chi}\to M^{s}_{P_{\Xi}}.$$ 
As in (\cite{Koo1},\cite{Koo2}) we have that the above morphism on the closed points factors through a surjection onto the torus fixed point locus given by $(M^{s}_{P_{\Xi}})^{T}_{cl}$ and we claim that fixing the equivariant structure for each stable sheaf we can make this morphism into an isomorphism.\par
We introduce the framing of a characteristic function as follows in order to fix the equivariant structure. For a fixed $(b_1,.,b_l)$, fix a top cone say $\sigma_1$ and choose the maximal $A_1,..,A_{d}$ corresponding to $_{b_1}\chi^{\sigma_{1}}$ and set each of them equal to $0.$ Stable sheaves being simple, all results in (\cite{Koo2}) for simple sheaves hold in our case and hence they hold for stable sheaves. Fixing the framing we observe keeping the underlying sheaf fixed, two equivariant structures differ by twisting with a character of the torus and hence can be realized as (\cite{FLT}) the kernel of the natural map below  
forgetting the natural equivariant structure given by,
\begin{equation}\label{ex}
0\to M\to Pic_{\mathbb T}(\sX) \to Pic({\sX}) \to 0     
\end{equation} where $M=Hom(\mathbb T, \mathbb C^{*}).$ Applying results of (\cite{JT},\cite{YJ1}) we see that the claim as in (\cite{Koo2}, eq. 14) holds in our case. Denote the morphism induced by the forgetful map as $F$ and denote the subset of framed characteristic sheaves corresponding to the Hilbert polynomial as $(\chi_{P_{\Xi}})^{fr}.$ Then the following holds,
\begin{thm}\label{claim}
$$F:\coprod\limits_{\vec\chi \in (\chi_{P_{\Xi}})^{fr}}M^{s} _{\vec\chi}\to M^{s}_{P_{\Xi}},$$ induces a bijective map on closed points onto $(M^{s}_{P_{\Xi}})^{\mathbb T}_{cl}.$  
\end{thm}
\begin{proof}
The above discussion suffices.   
\end{proof}
The rest of the section is devoted to understanding the Gieseker modified stable sheaves, and the technical details mostly follow from \cite{Koo2} in our case. We mention the results.
Next we show that the above morphism $F$ not only induces a bijection onto the fixed point locus on the closed points but also on the level of isomorphism of quasi-projective schemes.\par
\begin{thm}
Let $\sX$ be a smooth toric DM stack. Let $H$ be an ample line bundle on $X,$ and $\Xi$ be an equivariant generating sheaf on $\sX.$ Let $P_{\Xi}$ be a choice of a Hilbert polynomial of degree $d.$ Let $\vec\chi \in \chi_{P_{\Xi}}$ be a characteristic function and assume we can choose an equivariant line bundle matching Gieseker and GIT stability. Then there is a natural isomorphism of quasi-projective $\mathbb{C}$ schemes of finite type given by
$$(M^s_{P_{\Xi}})^\mathbb{T}\ \cong \coprod_{\vec\chi \in (\chi_{P_{\Xi}})^{fr}}M^s_{\vec\chi}.$$
\end{thm}
\begin{proof}
First note that $F_{cl}$ in the above theorem is bijective on the closed points and the left hand side is a disjoint union over a countable set. The right hand side is a $\mathbb{C}$ scheme of finite type. This implies there are only finitely many characteristic functions for which $M^s_{\vec\chi}$ are non-empty. This implies the LHS is a quasi-projective $\mathbb{C}$ scheme.\par
Now we want to apply (Prop 4.7 in \cite{Koo2}) to the above morphism $F$ and to the closed subscheme given by the inclusion $i:(M^s_{P_{\Xi}})^{\mathbb{T}}\to M^s_{P_{\Xi}}.$ Proceeding by induction on the length of local Artinian complex algebras with the residue field $\mathbb{C}$ the above (\ref{claim}) shows us that (Prop 4.7 in \cite{Koo2}) is satisfied for length 1.\par 
The rest of the proof shows that the hypothesis of (Prop 4.7) is satisfied for local Artinian algebra $A'$ of length $l+1$ provided it is true for local Artinian algebras of smaller length. (Prop 4.8 in \cite{Koo2}) achieves this by showing $$\coprod_{\vec\chi\in (\chi_{P_{\Xi}})^{fr}}\mathscr{M}^s_{\vec\chi}(A')\to \mathscr{M}^s_{P_{\Xi}}(A')$$ maps bijectively onto the fixed point locus $$(\mathscr{M}^s_{P_{\Xi}}(A'))^{\mathbb{T}_{cl}}.$$ The rest of the proof follows as in (\cite{Koo2}).

\end{proof}
We state the final result on modified Gieseker stable sheaves with a fixed modified Hilbert polynomial as follows,
\begin{thm}\label{Theorem 4.5}
Let $\sX$ be a smooth projective toric DM stack with $X$ as its projective normal coarse moduli space and with fixed ample line bundle, $\sO_{X}(1),\Xi$ a generating sheaf with fixed equivariant structure, and let $P_{\Xi}$ be a choice of a modified Hilbert polynomial. Then we have a canonical isomorphism of quasi-projective $\mathbb C$ schemes between the fixed point locus and the disjoint union of properly GIT stable moduli spaces with framed characteristic functions corresponding to the Hilbert polynomial given as,
$$(M^{s}_{P_{\Xi}})^{\mathbb T} \cong \coprod \limits_{\vec\chi\in (\chi_{P_{\Xi}})^{fr}}M^{s}_{\vec\chi}.$$
\end{thm}
 \begin{proof}
The proof follows the same procedure as (\cite{Koo2}). Indeed, the above theorem and (\ref{Gieseker}) imply the result we have to achieve.
\end{proof}
Next in spirit of (\cite{Koo2}, Thm 4.15) we obtain analogous results for the moduli spaces of modified slope-stable reflexive sheaves with fixed modified Hilbert polynomial and the same polarization data as above, on $\sX.$\par
Denote by $\chi^r$ the collection of characteristic functions associated to reflexive equivariant sheaves on $\sX$ which is a subset of the characteristic functions associated to equivariant torsion free sheaves on $\sX.$ Using (\ref{Reflexive hull}) we see that the stable reflexive sheaves are determined by one box summand per chart corresponding to the top cones.\par 
The modified slope (semi)stability for torus equivariant reflexive sheaves on $\sX$ can be checked for any equivariant coherent subsheaf from (\cite{Koo2}, Prop 4.13). For torsion free sheaves, such result follows from \cite{BDGP}.
The moduli functor $\mathscr{M}^{\mu s}_{P_{\Xi}}$ representing the family of geometrically modified slope-stable and reflexive sheaves as fibers defines an open property and hence is co-represented by an open subscheme $M^{\mu s}_{P_{\Xi}} \subseteq M^{s}_{P_{\Xi}}.$ \par
The space $M^{\mu s}_{\vec\chi}$ is constructed by taking GIT quotients of
a locally closed subscheme $$M^r_{\vec\chi}\subseteq A:=\prod\limits_{j=1}^{n}\prod\limits_{k=1}^{r-1}Gr(k,r)$$ with respect to the line bundle $$\mathscr{L}^{\mu}\in Pic^{SL(r,\mathbb{C})}(M^r_{\vec\chi}).$$ The construction of the line bundle matching properly GIT stable points of $M^r_{\vec\chi}$ to the modified slope-stable reflexive sheaves with characteristic function $\vec\chi$ follows exactly from (Theorem \ref{slope stability}.) The difference in the characteristic functions of a torsion free sheaf and a reflexive sheaf is given by the difference in their shapes. That is, a torsion free sheaf might not have strict corners in the filtration defining the characteristic function, unlike those of the reflexive sheaves which explains the ambient scheme $A$.\par
Denoting the collection of framed characteristic functions of equivariant reflexive sheaves on $\sX$ giving rise to the modified Hilbert polynomial $P_{\Xi}$ by $(\chi^{r}_{P_{\Xi}})^{fr}:=\chi^{fr}_{P_{\Xi}}\cap \chi^r_{P_{\Xi}},$ we have a similar result. The proof is analogous to the torsion free case and follows the proof of (\cite{Koo2}).
\begin{thm}\label{Final2}
Let $\sX$ be a smooth projective toric DM stack with $X$ as its projective normal coarse moduli space, an ample line bundle $H$ and a generating sheaf $\Xi$ with a fixed equivariant structure. Let $P_{\Xi}$ be a choice of a modified Hilbert polynomial of a reflexive sheaf. Then we have a canonical isomorphism of quasi-projective $\mathbb C$ schemes between the fixed point locus and the disjoint union of properly GIT stable moduli spaces with framed characteristic functions corresponding to the Hilbert polynomial given as,
$$(M^{\mu s}_{P_{\Xi}})^{\mathbb T} \cong \coprod \limits_{\vec\chi\in (\chi^r_{P_{\Xi}})^{fr}}M^{\mu s}_{\vec\chi}.$$
\end{thm}
\begin{example}
As an application of the results, we compute the generating function for the modified Euler characteristics of the stable torsion free modified slope sheaves of rank 1 on any projective smooth surface toric orbifold (i.e dimension=2), with the Picard group $\mathbb{Z}^{n-2}$ corresponding to a fixed $c_1=(x_1,..,x_{n-2}) \in \mathbb{Z}^{n-2}.$ We denote the reflexive hull corresponding to $c_1$ by $L_{(B_1,..,B_n)}$ using the surjectivity of (\ref{ex}). \par 
Fix $\mathscr{F}$ with a reflexive hull given by $L_{(B_1,..,B_n)}$ the toric line bundle. We are only interested in parametrizing the underlying torsion free sheaf and not the equivariant structure. Fixing $(B_1,B_2)$ on the top cone $\sigma_1$ we can express the other integers $(B_{3},..,B_n)$ as linear combinations of $(B_1,B_2).$ In that case considering $$\mathscr{F}\otimes L_{(-B_1,-B_2,\ \sum_{a=1}^{2}\lambda_a^3B_a\ ,..,\ \sum_{a=1}^{2}\lambda_a^{n}B_a)}$$ we keep the same torsion free sheaf and only change the equivariant structure. Hence it is enough to consider $\mathscr{F}$ with a reflexive hull $L_{(0,0,B_{3},..,B_{n})}$ and with the same fixed $c_1.$\par
Let $\mathscr{F}$ be rank 1 torsion free sheaf and denote the cokernel sheaf $\mathscr{Q}:=L_{(0,0,B_{3},..,B_n)}/\mathscr{F}.$ Under the torus action we have $l$ fixed points (zero dimensional substacks) corresponding to each top cone and denote them by $P_i$. The support of $\mathscr{Q}$ is given as follows. Let us denote by $j^i_a$ the number of skyscraper sheaves on the chart $\mathscr{U}_{\sigma_i}$ with support $P_i$ with the $K-$ group class given by $[\sO_{P_i}\otimes \alpha]$ where $\alpha \in DG_{\sigma_i}.$ \par 
Denote the modified Euler characteristic of the line bundle $(x_1,..,x_{n-2})$ by $\chi_{\Xi}(x_1,..,x_{n-2})$, which is the constant part of the modified Hilbert polynomial. Similarly, denote the modified Euler characteristic of 
$[\sO_{P_i}\otimes \alpha]$ by $\chi_{\Xi,\alpha}^{i}.$
We denote the generating function of the modified Euler characteristic of the moduli of modified stable rank 1 torsion free sheaves with fixed $c_1$ as,
$$Z_{c_1}(q)=\sum_{\chi_{\Xi}\in \mathbb{Z}}e(M_{\sX}(1,c_1,\chi_{\Xi}))q^{\chi_{\Xi}}=\sum_{\chi_{\Xi}\in \mathbb{Z}}e(M_{\sX}(1,c_1,\chi_{\Xi})^{\mathbb{T}})q^{\chi_{\Xi}}.$$
We use torus localization for the right hand side.\par  Using the fixed point locus description, we find that the closed points of $M_{\sX}(1,c_1,\chi_{\Xi})^{\mathbb{T}}$ are in bijection with $(j^i_{\alpha})$ where $\alpha \in DG_{\sigma_{i}}$ such that 
$$\chi_{\Xi}(x_1,..,x_{n-2})-\sum_{i=1}^{l}\sum_{\alpha\in DG_{\sigma_i}}j^{i}_{\alpha}\chi_{\Xi,\alpha}^{i}=\chi_{\Xi}.$$
The final formula follows
$$Z_{(x_1,..,x_{n-2})}(q)=q^{\chi_{\Xi}(x_1,..,x_{n-2})}\prod_{k=1}^{\infty}\prod_{i=1}^{l}\prod_{\alpha \in DG_{\sigma_i}}\frac{1}{(1-q^{-\chi^i_{\Xi,\alpha}k})^{j^i_{\alpha}}}.$$
\end{example}
\begin{example}
For this example, see \cite{GJK}.
In the case of weighted projective surface orbifold $\sX:=\mathbb{P}(a,b,c)$ one can show that the modified Euler characteristic of a sheaf corresponding to a single $\mathbb{C}$ in one chart depends on the size of each cell in the chart itself (i.e the Hilbert polynomial does not depend on the fine-grading.) The generating sheaf is given by $\Xi=\bigoplus_{k=1}^{abc-1}\sO_{\sX}(k).$ Here $\sO_{\sX}(1)$ corresponds to the standard representation and is the generator of $Pic({\sX})\cong \mathbb{Z}.$ Thus for a fixed $c_1=x \in \mathbb{Z},$ we have,
$$Z_{x}(q)=q^{\chi_{\Xi}(\sO_{\sX}(x))}\prod_{k=1}^{\infty}\frac{1}{(1-q^{-abk})(1-q^{-bck})(1-q^{cak})}$$ and
$$\chi_{\Xi}(\sO_{\sX}(x))$$\\
$$=\frac{a^2+b^2+c^2+3ab+3bc+3ca+6((a+b+c)\sum_{k=0}^{abc-1}(x+k)+\sum_{k=0}^{abc-1}(x+k)^2)}{12}.$$
\end{example}




\subsection*{}

\end{document}